\documentclass[twoside,english]{article}
\usepackage[T1]{fontenc}
\usepackage{geometry}
\usepackage{color}
\usepackage{babel}
\usepackage{array}
\usepackage{verbatim}
\usepackage{float}
\usepackage{mathtools}
\usepackage{bm}
\usepackage{multirow}
\usepackage{amsmath}
\usepackage{amsthm}
\usepackage{amssymb}
\usepackage{graphicx}
\usepackage{booktabs}
\usepackage{todonotes}
\usepackage{algorithm}
\usepackage{algorithmic}
\usepackage{url}
\usepackage{authblk}
\usepackage[round]{natbib}
\usepackage{hyperref}
\theoremstyle{plain}
\newtheorem{theorem}{Theorem}[section]

\newtheorem{lemma}[theorem]{Lemma}
\newtheorem{corollary}[theorem]{Corollary}
\theoremstyle{definition}

\newtheorem{assumption}[theorem]{Assumption}
\theoremstyle{remark}
\newtheorem{remark}[theorem]{Remark}

\global\long\def\efosp{FOSP}%
\global\long\def\esosp{SOSP}%
\global\long\def\drsopo{DR-SOPO}%
\global\long\def\dvrsopo{DVR-SOPO}%

\newcommand{\reals}{{\mathbb{R}}}

\newcommand{\trans}{{\textup{T}}}

\newcommand{\Acal}{{\mathcal{A}}}

\newcommand{\Mcal}{{\mathcal{M}}}

\newcommand{\Ocal}{{\mathcal{O}}}
\newcommand{\Pcal}{{\mathcal{P}}}
\newcommand{\Rcal}{{\mathcal{R}}}
\newcommand{\Scal}{{\mathcal{S}}}

\global\long\def\norm#1{\lVert#1\rVert}%

\newcommand{\Expect}{\mathop{\mathbb E{}}\nolimits}
\newcommand{\Prob}{\mathop{ \rm Pr}}
\newcommand{\Var}{\mathop{\rm Var}}

\newcommand{\del}{\,\mathrm{d}}

\newcommand{\vep}{\varepsilon} %

\global\long\def\Brbra#1{\Big(#1\Big)}%
\global\long\def\sbra#1{[#1]}%
\global\long\def\bsbra#1{\big[#1\big]}%
\global\long\def\cbra#1{\{#1\}}%

\begin{document}
\title{Stochastic Dimension-reduced Second-order Methods for Policy Optimization\footnote{Jinsong Liu and Chenghan Xie contribute equally and are listed in alphabetical order}}

\author[1]{Jinsong Liu\thanks{liujinsong@163.sufe.edu.cn}}
\author[2]{Chenghan Xie\thanks{20307130043@fudan.edu.cn}}
\author[1]{Qi Deng\thanks{qideng@sufe.edu.cn}}
\author[1]{Dongdong Ge\thanks{ge.dongdong@mail.shufe.edu.cn}}
\author[3]{Yinyu Ye\thanks{yyye@stanford.edu}}
\affil[1]{Shanghai University of Finance and Economics}
\affil[2]{Fudan University}
\affil[3]{Stanford University}

\maketitle

\begin{abstract}
In this paper, we propose several new stochastic second-order algorithms for policy optimization that only require gradient and   Hessian-vector product in each iteration, making them computationally efficient and  comparable to policy gradient methods. Specifically, we propose a dimension-reduced second-order method (DR-SOPO) which repeatedly solves a projected two-dimensional trust region subproblem. We show that DR-SOPO obtains an $\mathcal{O}(\epsilon^{-3.5})$ complexity for reaching approximate first-order stationary condition and certain subspace second-order stationary condition. In addition, we present an enhanced algorithm (DVR-SOPO) which further improves the complexity to $\mathcal{O}(\epsilon^{-3})$ based on the variance reduction technique. Preliminary experiments show that our proposed algorithms perform favorably compared with stochastic and variance-reduced policy gradient methods. 
\end{abstract}

\section{Introduction}
The policy gradient (PG) method, pioneered by \citet{williams1992simple}, is a widely used approach for finding the optimal policy in reinforcement learning (RL). The main idea behind PG is to directly maximize the total reward by using the (stochastic) gradient of the cumulative rewards. PG is particularly useful for  high dimensional continuous state and action spaces due to its ease in implementation. In recent years, the PG method has gained much attention due to its significant empirical successes  in a variety of challenging  RL applications~\citep{lillicrap2015continuous, silver2014deterministic}.

 Despite its wide application and popularity,  there is a growing concern about the high variance of classical PG method \cite{williams1992simple} and its variants \cite{sutton2000empirical, baxter2001infinite}. 
 To address this issue, various approaches such as trust region method~\citep{schulman15trpo}, proximal policy optimization~\citep{schulman2017proximal} and natural PG~\citep{Kakade01npg} have been proposed to further improve the algorithm robustness and efficiency.
Motivated by the variance reduction technique developed in stochastic optimization, (e.g. \citet{johnson2013accelerating, fang2018spider, zhou2020stochastic}), recent work \citep{papini2018stochastic, Xu2020Sample, RF1}  have developed new variance-reduced
policy gradient methods for model-free RL. In particular, \citet{Xu2020Sample, RF1} obtain $\Ocal(\epsilon^{-3})$ sample complexity\footnote{For sample complexity, we mean the total number of samples involved in the estimation of gradient and hessian-vector product} for finding an
$\epsilon$-first order stationary point, which improves the  $\Ocal(\epsilon^{-4})$ complexity  of vanilla  PG method~\citep{williams1992simple} by a factor of $\Ocal(\epsilon^{-1})$. 
We address that all those variance-reduced PG methods only have guaranteed convergence to an $\epsilon$ first-order stationary point ($\epsilon$-FOSP), which, due to the non-convexity of RL objective,  is insufficient to ensure  global optimality. By exploiting the objective landscape in policy optimization, much recent efforts have been directed towards finding global solutions beyond first-order stationary convergence. 
For example, see \citet{agarwal2021theory, bhandari2019global, cen2022fast, lan2022policy} on global convergence of various PG methods and \citet{zhang2021convergence, liu2020improved} for the variance-reduced PG methods.
Despite much progress, the global optimality can be achieved only for  special parameterized settings such as softmax tabular policy and Fisher-non-degenerate policy, or when function approximation error is controllable~\citep{yuan2022general}. 

For the more general setting,  \citet{zhang2020global, yang2021sample} have established the convergence of  PG methods  to second-order stationary point (SOSP), which successfully exclude saddle point solutions that have indefinite Hessian. This research line is driven by the recent progress on developing efficient (stochastic) gradient methods to escape saddle-point solution~\citep{daneshmand2018escaping, jin2017escape}.
In addition to the first order methods,  second-order methods, such as cubic regularized Newton method~\citep{nesterov2006cubic, cartis2011adaptive,  tripuraneni2018stochastic, kohler2017sub} and trust region method~\cite{yuan2015recent, curtis2017trust, curtis2020fully}, are known to have provable finite time convergence to SOSP solutions in nonconvex (and stochastic) optimization.
However, a challenge to the second-order methods is that they involve some nontrivial subproblems which require either matrix decomposition or external iterative solvers to obtain desired solutions. For example, see \citep{nesterov2006cubic, agarwal2017finding, nocedal1999numerical}. 
To alleviate this issue,  \citet{wang2022stochastic} proposed a new stochastic cubic-regularized policy gradient method~(SCP-PG) which only involves the cubic regularized Newton subproblem when converging to FOSP. While SCP-PG can save the additional Hessian-vector product from time to time,  it has an unsatisfactory sample complexity of $\Ocal(\epsilon^{-7/2})$, which is worse than the optimal rate of $\Ocal(\epsilon^{-3})$ \citep{arjevani2020second}. 

In this paper, we aim to address some remaining issues in designing efficient second-order method for model-free policy optimization. 
Our work is motivated by the recent work in trust region optimization, and particularly, the dimension reduced second-order method (DRSOM) for nonconvex optimization~\citep{SecondOrderOptimizationAlgorithmsI, zhang2022drsom}. DRSOM has a significant advantage in that it solves a relatively simple 2-dimensional trust region subproblem without the need of  external solvers, resulting a computation overheads on a par with gradient descent methods. Building on this research direction, we develop several new stochastic second-order methods for policy optimizations.
Our main contribution can be summarized as the following aspects.  
\begin{itemize}
    \item First, we propose a stochastic second-order method for policy optimization (\drsopo{}), based on the recently proposed dimension-reduced trust region method. 
    We show that the proposed method achieves a convergence rate of $\Ocal(\epsilon^{-3.5})$ for approximating a first order stationary point and a second-order stationary point in certain subspace.
    Our method only involves a cheap dimension reduced subproblem and hence bypasses the computation burden in solving a full dimensional quadratic program for trust region method. As a consequence, the computational overheads of \drsopo{} is comparable to that of policy gradient. 
   \item  Second, to further enhance the  efficiency of \drsopo{},  we propose a dimension and variance-reduced second-order method for policy optimization (\dvrsopo{}). \dvrsopo{} incorporates a Hessian-aided variance reduction technique~\citep{RF1} to \drsopo{}. The variance technique significantly improves the complexity in terms of the stochastic gradient queries, which is the dominating term in the complexity bound. We show that~\dvrsopo{} further improves the convergence rate of \drsopo{} to $\Ocal(\epsilon^{-3})$.
   \item  Finally, we provide some preliminary experiments  to demonstrate the empirical advantage of our proposed methods against policy gradient method and variance-reduced policy gradient method. Our theoretical analysis and empirical study show the great potential of directly using second-order information in policy optimization and reinforcement learning.
\end{itemize}

\paragraph{Comparison with TRPO} Prior to our work, \citet{schulman15trpo} proposed the trust region policy optimization (TRPO) method, which has received great popularity in RL. However, it is important to note that  TRPO is a first-order method that uses KL divergence to enforce optimization stability. This approach fundamentally differs from our work, which draw inspiration from trust region method in numerical optimization. Typically, a classic trust region method involves $\ell_p$-ball constraints to improve the convergence of second-order method. Recently,  \citet{jha2020quasi} replaced the linear objective in TRPO by quadratic approximation, and proposed  to  solve the modified TRPO by Quasi-Newton methods. However, their work did not provide any theoretical guarantee for the proposed method.

\subsection{More on Related Work}
We give a more detailed review over two research directions in policy optimization that are closely related to our work:  the study on the convergence  to \esosp{} and the variance reduction technique for the PG methods.

\paragraph{Convergence to SOSP.}  \citet{zhang2020global} appears to be the first study showing that PG can escape FOSP and reach high-order stationary solutions. Specifically, they proposed a random horizon rollout for unbiased estimation of policy gradient. Furthermore, equipped with a periodically enlarged stepsize and the correlated negative curvature technique, they show that the a modified PG method an $ \Ocal(\epsilon^{-9})$ sample complexity for converging to an $(\epsilon,\sqrt{\epsilon})$-SOSP. In a follow-up work~\cite{yang2021sample}, the authors improved the sample complexity to $\Ocal(\epsilon^{-9/2})$ and extended the SOSP analysis to extensive policy optimization algorithms, though under a restrictive objective structure assumption. A more recent work \citep{wang2022stochastic} proposed to use stochastic cubic Newton method for policy optimization and further improved the sample complexity to  $\Ocal(\epsilon^{-7/2})$.

\paragraph{Variance-reduction.} 
Recent work in policy optimization has made a strong effort to apply variance reduction (VR) techniques, inspired by their success in the oblivious stochastic setting, such as SARAH~\citep{nguyen2017sarah} and SPIDER~\citep{fang2018spider}.
Hessian-aided VR \citep{RF1} is one of the examples that succeed in non-oblivious setting. Besides, \citet{Xu2020Sample} proposes a SRVR-PG method based on SARAH and use important sampling techniques to deal with the distribution shift, while \citet{pham2020hybrid} provides a single-loop version by a hybrid approach. Based on a novel gradient truncation technique, \citet{zhang2021convergence} successfully removes a strong assumption on the variance bound of important sampling in earlier work. All of the above algorithms match the sample  complexity of $\mathcal{O}(\epsilon^{-3})$. Motivated by the success of stochastic recursive momentum (i.e. STORM, \citet{cutkosky2019momentum}), \citet{huang2020momentum, yuan2020stochastic} propose new STORM-like PG methods that blend momentum in the updates. These methods have the same $\mathcal{O}(\epsilon^{-3})$ complexity but do not require  alternating between large and small batch sizes. 

\paragraph{Paper structure}
Our paper proceeds as follows. Section~\ref{Preliminaries} introduces notations and background in policy optimization. Section~\ref{sec:dr} presents the dimension reduced second-order method and its convergence analysis. Section~\ref{sec:dvr} presents the dimension and variance-reduced method based on the Hessian-aid VR technique. Section~\ref{sec:implement} elucidates more  details about the implementation of trust region procedure. Section~\ref{sec:numerical} conducts empirical study to show the advantage of our proposed methods.  We draw conclusion in Section~\ref{sec:conclusion} and leave technical proof in the appendix sections. 

\section{Preliminaries}
\label{Preliminaries}
\paragraph{Notation}
For a square matrix $A\in\reals^{n\times n}$, we define norm for matrix as $\|A\|=\sqrt{\lambda_M}$, where $\lambda_M$ is the  eigenvalue of $A^T A$ with biggest absolute value. For a vector $v\in\reals^n$, we use $\|v\|$ to express the standard Euclidean norm. $\|v\|_Q\coloneqq \sqrt{v^\trans Q v}$ where $Q$ is a positive-definite matrix.

Consider the Markov decision process $\mathfrak{M} =\cbra{\Scal, \Acal, \Pcal, r,\gamma}$. Here, $\Scal$ is the state space, $\Acal$ is the action space, $\mathcal{P}\left(s_{h+1} \mid s_h, a_h\right)$ is the transition kernel, $r:\Scal\times \Acal\rightarrow \reals$  is the regret/cost function\footnote{We use regret as opposed to the reward function in  standard MDP literature. Our development is  more aligned with optimization literature and  results in minimization as opposed to maximization problem.} and $\gamma\in[0,1)$ is a discount function. 
The agent's behavior is modeled by policy function $\pi$, and $\pi(a\mid s)$ is the density of action $a\in\Acal$ given the state $s\in\Scal$. We describe the policy $\pi_\theta$ to reflect the fact the policy function is parameterized by a vector $\theta\in\reals^d$. 
 Let $s_0$ follow from the initial distribution $\rho\left(\cdot\right)$. Let $H$ be the length of truncated trajectory, and $p\left(\tau ; \pi_\theta\right)$ be the density of trajectory $\tau=(s_0, a_0, \ldots, s_{H-1}, a_{H-1})$ following from the MDP:
$
p(\tau ; \pi_\theta):=\rho(s_0) \prod_{h=0}^{H-1} \mathcal{P}(s_{h+1} \mid s_h, a_h)\, \pi_\theta(a_h \mid s_h) .
$
For brevity, we use the notation $p(\tau;\pi_\theta)$ and $p(\tau; \theta)$ interchangeably. 
We consider the accumulated discounted regret:
$$
\mathcal{R}(\tau):=\sum_{h=0}^{H-1} \gamma^h r(s_h, a_h) .
$$
We remark that for more general objective with infinite horizon, one can set $H=\mathcal{O}(\log(\epsilon^{-1}))$ and view $\mathcal{R}(\tau)$ as a truncated estimator. See more discussion in Appendix~\ref{sec:infinite}. Our goal is to solve the following policy optimization problem which  minimizes the expected discounted trajectory regret/cost
\begin{equation}\label{eq:policy-optimize}
\min_{\theta \in \mathbb{R}^d} J(\theta):=\Expect_{\tau}[\Rcal(\tau)]=\int \mathcal{R}(\tau)\, p(\tau ; \pi_\theta) \,\rm{d} \tau
\end{equation}
Note that the above problem can be viewed as nonconvex stochastic optimization (e.g. \citet{ghadimi2013stochastic}).
We say that a solution $\theta$ is an $(\vep_1, \vep_2)$-second-order stationary point (($\vep_1, \vep_2$)-\esosp{}) if i) $\norm{\nabla J(\theta)} \le \vep_1$ and ii) $\lambda_{\min}(\nabla^2 J(\theta)) \geq -\vep_2$. Here $\lambda_{\min}$ denotes the smallest eigenvalue. 
Moreover, $\theta$ is an $\vep_1$-first-order stationary point ($\vep_1$-\efosp{}) if only condition i) holds. 
An unbiased estimator of the  gradient $\nabla J(\theta)$  is given by
\begin{equation}
\begin{aligned}
g(\theta ; \tau)&:= \sum_{h=0}^{H-1} \Psi_h(\tau) \nabla \log \pi_\theta(a_h \mid s_h)\label{eq:sample-grad}
\end{aligned}
\end{equation}
where we denote $\Psi_h(\tau)=\sum_{i=h}^{H-1} \gamma^i r(s_i, a_i)$.
Due to $\nabla J({\theta})=\Expect_{\tau \sim p(\cdot;\theta)}[g(\theta;\tau)]$ 
 we can derive an unbiased estimator of the  Hessian   $\nabla^2 J(\theta)$ by
\begin{equation}\label{eq:stoc-hessian}
H(\theta ; \tau):= \nabla g(\theta;\tau)+g(\theta;\tau) \nabla \log p(\tau; {\theta})^\trans.
\end{equation}
Let $\mathcal{M}$ be a set of i.i.d. trajectories sampled from density $p\left(\cdot ; \pi_\theta\right)$.
We denote $g(\theta ; \mathcal{M})=\frac{1}{|\mathcal{M}|} \sum_{\tau \in \mathcal{M}} g(\theta;\tau)$ and
 $H(\theta;\Mcal)=\frac{1}{|\mathcal{M}|} \sum_{\tau \in \mathcal{M}} H(\theta;\tau)$.

\section{DR-SOPO}\label{sec:dr}
\subsection{Algorithm}
\begin{table*}[htbp]
\small
    \caption{Comparison of  different policy optimization algorithms. We ignore the logarithmic terms, for simplicity.}\label{tab:complexity-compare}
    \centering
    \begin{tabular}{c|l|c|c}
        \hline
        Algorithms & Conditions  & Guarantee & Sample complexity \\
        \hline
        \hline
        REINFORCE~\citep{williams1992simple}  & Assumption \ref{assm:bdd-reward},\ref{assm:regular}  & FOSP  & $\mathcal{O}(\epsilon^{-4})$     \\
        \hline
         SRVR-PG~\citep{Xu2020Sample} &  \begin{tabular}[c]{@{}l@{}}Assumption \ref{assm:bdd-reward},\ref{assm:regular}\\ $\mathbb{\operatorname {Var}}[\prod_{i \geq 0} \frac{\pi_{\theta_0}(a_i \mid s_i)}{\pi_{\theta_t}(a_i \mid s_i)}]<\infty$\end{tabular} 
          & FOSP  & $\mathcal{O}(\epsilon^{-3})$\\
         \hline
         MBPG~\citep{huang2020momentum}  & \begin{tabular}[c]{@{}l@{}}Assumption \ref{assm:bdd-reward},\ref{assm:regular}\\ Bounded variance of IM samplings \end{tabular} & FOSP   & $\Ocal(\epsilon^{-3})$ \\
         \hline
        TSIVR-PG~\citep{zhang2021convergence}  & Assumption \ref{assm:bdd-reward},\ref{assm:regular} & FOSP   & $\mathcal{O}(\epsilon^{-3})$ \\
         \hline 
         HAPG~\citep{RF1}&  Assumption \ref{assm:bdd-reward},\ref{assm:regular} & FOSP  & $\mathcal{O}(\epsilon^{-3})$\\
         \hline
         MRPG~\citep{zhang2020global}  &\begin{tabular}[c]{@{}l@{}}Assumption \ref{assm:bdd-reward},\ref{assm:regular}, \ref{assm:lip-hessian}\\ Positive-definite Fisher information \end{tabular} & SOSP   & $\Ocal(\epsilon^{-9})$ \\
         \hline 
         \citet{yang2021sample}  &\begin{tabular}[c]{@{}l@{}}Assumption \ref{assm:bdd-reward},\ref{assm:regular}, \ref{assm:lip-hessian}\\ 
         All saddle points are strict \end{tabular}  & SOSP   & $\mathcal{O}(\epsilon^{-\frac{9}{2}})$ \\
         \hline 
          SCR-PG~\cite{wang2022stochastic}  &  \begin{tabular}[c]{@{}l@{}}Assumption \ref{assm:bdd-reward},\ref{assm:regular} \ref{assm:lip-hessian}\\ Bounded variance of IM samplings \end{tabular}  & SOSP  & $\mathcal{O}(\epsilon^{-\frac{7}{2}})$ \\
         \hline 
         DR-SOPO  & Assumption \ref{assm:bdd-reward},\ref{assm:regular} \ref{assm:lip-hessian} \ref{assm:drsom} & SOSPS   & $\mathcal{O}(\epsilon^{-\frac{7}{2}})$ \\
         \hline 
         DVR-SOPO  &Assumption \ref{assm:bdd-reward},\ref{assm:regular} \ref{assm:lip-hessian} \ref{assm:drsom} & SOSPS   & $\mathcal{O}(\epsilon^{-3})$ \\
         \hline 
    \end{tabular}
\end{table*}

We first present the dimension-reduced second-order method for policy optimization (DR-SOPO), which extends dimension-reduced trust region technique in DRSOM to the non-oblivious stochastic setting. 
The novel ingredient in DRSOM is to determine the descent step by  involving a two-dimensional trust-region subproblem, which is much simpler than the full dimensional quadratic program in standard trust region method. More specifically, let $d_t=\theta_t-\theta_{t-1}$, $g_t=g(\theta_t;\Mcal_g)$, $H_t=H(\theta_t;\Mcal_H)$, instead of  updating $\theta$ by the gradient descent step, we update 
\[\theta_{t+1}=\theta_t-\alpha_t^1 g_t+\alpha_t^2 d_t,\] 
where the step size $\alpha_t=(\alpha_t^1, \alpha_t^2)^\trans$ is determined 
by solving the following dimension-reduced trust-region (DRTR) problem:
\begin{equation}\label{eq:DRTR}
\begin{aligned}
&\min_{\alpha \in \mathbb{R}^2} && m_t(\alpha):=J\left(\theta_t\right)+c_t^\trans \alpha+\frac{1}{2} \alpha^\trans Q_t \alpha \\
&\   \textup{ s.t.} && \|\alpha\|_{G_t} \leq \Delta,
\end{aligned}
\end{equation}
with
$$
Q_t=\begin{bmatrix}
g_t^\trans H_t g_t & -d_t^\trans H_t g_t \\
-d_t^\trans H_t g_t & d_t^\trans H_t d_t
\end{bmatrix} \in \mathcal{S}^2,\\$$
$$
c_t:=\begin{pmatrix}
-\|g_t\|^2 \\
g_t^\trans d_t
\end{pmatrix},\ G_t=\begin{bmatrix}
g_t^\trans g_t & -g_t^\trans d_t \\
-g_t^\trans d_t & d_t^\trans d_t
\end{bmatrix},
$$
and $\|\alpha\|_{G_t}=\sqrt{\alpha^\trans G_t \alpha}$. 

 The two-dimensional subproblem~\eqref{eq:DRTR} has a closed-form solution, which can be found in Appendix~\ref{Optimal condition and solution for trust region problems}. Note that while DRTR conceptually utilizes the curvature information, it can be implemented without explicitly computing the Hessian. In fact,  we only require two additional Hessian-vector products to formulate~\eqref{eq:DRTR}. Hence, DRTR only incurs a much lower computation cost than standard trust region method~\cite{nocedal1999numerical}. 

Moreover, the next lemma implies that while problem~\eqref{eq:DRTR} is a two-dimensional trust region model, it can be equivalently
transformed into a “full-scale” trust-region problem associated with projected Hessian $\tilde{H}_t$, which is pivotal to our theoretical analysis. 
\begin{lemma}\label{lmdrsom1}
    The subproblem (\ref{eq:DRTR}) is equivalent to 
\begin{equation}\label{subprob-equiv}
\begin{aligned}
\min_{d \in \mathbb{R}^n}& \ \tilde{m}_t(d):=J\left(\theta_t\right)+g_t^\trans d+\frac{1}{2} d^\trans \tilde{H}_t d \\
\textup{s.t.}& \ \|d\| \leq \Delta_t,
\end{aligned}
\end{equation}
where $\tilde{H}_t=V_t V_t^\trans H_t V_t V_t^\trans$ and $V_t$ is the orthonormal bases for $\mathcal{L}_t:=\operatorname{span}\left\{g_t, d_t\right\}$.
\end{lemma}
For proof, see \citet{zhang2022drsom}. Therefore, $\Tilde{H}_t$ can be viewed as an approximated Hessian matrix in
the inexact Newton method, and DRTR may similarly be regarded as a cheap quasi-Newton
method.
For the ease of notation, let $\tilde\nabla J(\theta_t)=V_tV_t^\trans \nabla J(\theta_t)V_tV_t^\trans$ be the projected Hessian onto the subspace $\mathcal{L}_t$.

With gradient estimator~\eqref{eq:sample-grad} and Hessian estimator~\eqref{eq:stoc-hessian}, we present the basic DR-SOPO  in Algorithm~\ref{alg:drsopo}, in which we use a fixed trust region radius to make the complexity analysis more concise. A more practical implementation of DR-SOPO will be presented in Algorithm~\ref{prac-drsopo}, which dynamically adjusts the trust region radius and gets better practical performance.

\begin{algorithm}[h]
	\caption{Basic DR-SOPO}\label{alg:drsopo}
	\begin{algorithmic}[1]
            \STATE Given $T$, $\Delta$
		\FOR{$t=1, \dots ,T$}
            \STATE  Collect sample trajectories  $\Mcal_g$ and compute $g_t$
            \STATE Collect sample trajectories  $\Mcal_H$ and compute $H_t$
            \STATE Compute  stepsize $\alpha_1,\alpha_2 $ by solving the DRTR problem \eqref{eq:DRTR}
            \STATE Update: $\theta_{t+1} \gets \theta_{t}-\alpha_1 g_t+\alpha_2 d_t$
            \ENDFOR
	    \end{algorithmic}
\end{algorithm}

\subsection{Complexity analysis}

DR-SOPO can find a second-order stationary point in subspace~(($\epsilon,\sqrt{\epsilon}$)-SOSPS), which is defined as $\theta$ such that 
\[
\|\nabla J(\theta)\| \leq c_1 \cdot \epsilon, \; \lambda_{\min}(\Tilde{\nabla}^2 J(\theta))\geq -c_2 \cdot \sqrt{\epsilon},
\]
where $\Tilde{\nabla}^2 J(\theta)$ is the projection of $\nabla^2 J(\theta)$ in the particular searching space. Next, we will show that DR-SOPO converges to a $(\epsilon,\sqrt{\epsilon})$-SOSPS with a total sample complexity of $\mathcal{O}(\epsilon^{-3.5})$.

Before developing the convergence analysis, we need to make a few assumptions.
\begin{assumption}\label{assm:bdd-reward}There exists a constant $R>0$ such that $
|r(s,a)| \leq R$ 
 for all $a \in \mathcal{A}$ and $s \in \mathcal{S}$.
\end{assumption}
\begin{assumption}
[Expected Lipschitzness and smoothness]\label{assm:regular} There exist constants $G>0$ and $L>0$ such that  
\begin{equation*}
    \begin{aligned}
        \Expect_\tau[\|\nabla \log \pi_\theta(a \mid s)\|^4] \leq G^4 ,\;
        \Expect_\tau[\|\nabla^2 \log \pi_\theta(a \mid s)\|^2] \leq L^2,
    \end{aligned}
\end{equation*}
for any choice of parameter $\theta$ and state-action pair $(s, a)$.
\end{assumption}
\begin{assumption}\label{assm:lip-hessian}
There exists $M>0$ such that for any $\theta_1,\theta_2\in\reals^n$
    $$||\nabla^2 J(\theta_1)-\nabla^2 J(\theta_2)||\leq M||\theta_1-\theta_2||$$
\end{assumption}
\begin{assumption}\label{assm:drsom}There exists a constant $C>0$ such that
$$||(\nabla^2 J(\theta_t)-\tilde{\nabla}^2 J(\theta_t))d_{t+1}||\leq MC\|d_{t+1}\|^2$$
\end{assumption}
\begin{remark}
Both Assumption \ref{assm:bdd-reward} and \ref{assm:regular} are fairly standard in the RL literature. Note that, however, \ref{assm:regular} is  weaker than the standard LS assumptions which are defined for $\nabla \log(\pi_\theta(a\mid s))$ and $\nabla^2 \log \pi_\theta(a\mid s)$ pointwisely. (e.g. \citet{papini2018stochastic, RF1}). More detailed comparison can be found in \citet{yuan2022general}.
  Assumption~\ref{assm:lip-hessian} on  Lipschitz Hessian  is standard for second-order methods, and for RL,  $M$ can be deduced by some other regularity conditions. For example, see~\cite{zhang2020global}. 
 Assumption~\ref{assm:drsom} is required for the convergence analysis of DRSOM~\cite{zhang2022drsom}, and is commonly used in the literature \citep{dennis1977quasi, cartis2011adaptive}.
\end{remark}
In view of Assumption~\ref{assm:regular}, we  characterize some properties of the stochastic
gradient and stochastic Hessian, which paves the way for our convergence analysis of the stochastic algorithms.
\begin{lemma}\label{lm LG-Variance-bound} 
Under  Assumptions  \ref{assm:bdd-reward} and \ref{assm:regular}, we have for all $\theta\in\reals^d$,
\[
\begin{aligned}
&\Expect_\tau[\|g(\theta ;\tau)-\nabla J(\theta)\|^2] \leq \frac{G^2 R^2}{(1-\gamma)^3}:=G_g^2, \\
&\Expect_\tau[\|H(\theta; \tau)-\nabla^2 J(\theta)\|^2]\leq \frac{H^4 G^4 R^2}{(1-\gamma)^2}+\frac{L^2 R^2}{(1-\gamma)^4}:=G_H^2.
\end{aligned}
\]
\end{lemma}
Note that our bound is slightly better than the one in \cite{RF1}, though under weaker Assumption~\ref{assm:regular}. Moreover, based on the fact that future policy at time $t^{\prime}$ can't affect reward at time $t$ if $t<t^{\prime}$, we can construct a variance-reduced 
unbiased Hessian estimator, whose variance bound doesn't depend on $H$. More details can be found in \ref{discuss with H estimator}.
Next, we derive variance bounds on the sampling gradient and Hessian in Algorithm~\ref{alg:drsopo}.
\begin{lemma}[Variance bounds on stochastic estimators]\label{lm sgd_var-bound for g-estimator}\label{lm var for hessian}
In Algorithm~\ref{alg:drsopo},  by setting $|\Mcal_g|=\frac{144G_g^2}{\epsilon^2}$ , we have
$$
\Expect_t\bsbra{\left\|g_t-\nabla J\left(\theta_t\right)\right\|^2} \leq \frac{\epsilon^2}{144M^2}.
$$
Moreover, by setting $\left|\mathcal{M}_H\right|=\frac{22\times24^2G_H^2\log(n)}{\epsilon}$, we have
$$
\Expect_t\bsbra{\left\|H_t-\nabla^2 J\left(\theta_t\right)\right\|^2} \leq \frac{\epsilon}{24^2}.
$$
where $\Expect_t$ denotes the  expectation conditioned on all the randomness before $t$-th iteration.
\end{lemma}

Based on Assumption \ref{assm:drsom}, we obtain a regularity condition such that Hessian estimator $H_t$ agrees
with $\tilde{H}_t$  along the directions $d_{t+1}$ formed by past iterates $\theta_t$.
\begin{lemma}\label{lm3.5}
    Under Assumption \ref{assm:drsom}, then we have
\begin{equation*}
    \Expect_t\bsbra{\|(H_t-\tilde{H}_t)d_{t+1}\|} \leq M\tilde{C} \Delta^2,
\end{equation*}
where $\tilde{C}=C+\frac{1}{24}$ and $d_t \leq \Delta=\frac{2\sqrt{\epsilon}}{M}$.
\end{lemma}

The following lemma implies a sufficient descent property in solving the trust region subproblem.
\begin{lemma}[Model reduction]\label{lm drsom2}
At the $t$-th iteration, let $d_{t+1}$ and $\lambda_t$ be the optimal primal and dual solution of~\eqref{subprob-equiv}. We have the following amount of decrease on $\tilde{m}_t$:
\[
\tilde{m}_t\left(d_{t+1}\right)-\tilde{m}_t(0)=-\frac{1}{2} \lambda_t \|d_{t+1}\|^2.
\] 
\end{lemma}

\par
With all the tools at our hands, we  derive the main convergence property of the \drsopo{} algorithm in the following theorem.
\begin{theorem}[Convergence rate of \drsopo{}]\label{thm sgd}
     Suppose Assumptions~\ref{assm:bdd-reward}-\ref{assm:drsom} hold.
     Let $\Delta_t=\Delta=\frac{2\sqrt{\epsilon}}{M}$ and $\Delta_J$ be a constant number that s.t. $\Delta_J\geq J(\theta_0)-J^*$,
     by setting $M_g=\frac{144G_g^2}{\epsilon^2}$,  $\left|\mathcal{M}_H\right|=\frac{22\times24^2G_H^2\log(d)}{\epsilon}$, $T=\frac{24M^2\Delta_J}{ \epsilon^{\frac{3}{2}}}$ in Algorithm~\ref{alg:drsopo}, we have
$$
\Expect[\left\|\nabla J\left(\theta_{\bar{t}}\right)\right\|] \leq \frac{(3+4\tilde{C})}{M}\epsilon ,\\
\Expect[\lambda_{\min}(\tilde{\nabla}^2J(\theta_{\bar{t}}))]\geq-3\sqrt{\epsilon},
$$
where $\bar{t}$ is sampled from $\{1, \ldots, T\}$ uniformly  at random.
Moreover, with probability at least $\frac{7}{8}$,  we have 
\[
\left\|\nabla J\left(\theta_{\bar{t}}\right)\right\| \leq \frac{(12+16\tilde{C})}{M}\epsilon.
\]
\end{theorem}
Using the above result, we develop the sample complexity of Algorithm~\ref{alg:drsopo} to approximate second-order stationary condition. 
\begin{corollary}[Sample complexity of \drsopo{}]\label{thm sgd sample}
Under all the assumptions and parameter settings in Theorem~\ref{thm sgd}, 
Algorithm~\ref{alg:drsopo} returns a point $\widehat{x}$ such that $\Expect[\|\nabla J(\widehat{x})\|] \leq \frac{(3+4\tilde{C})}{M}\epsilon$, $\Expect[\lambda_{\min}(\tilde{\nabla}^2 J(\widehat{x})]\geq -3\sqrt{\epsilon} $ after performing at most
\begin{equation}\label{eq:bound-drso}
\mathcal{O}\left(\frac{\Delta_J M^2G_g^2}{\epsilon^{3.5}} +\frac{\Delta_J M^2G_H^{2} }{\epsilon^{2.5}}\right)
\end{equation}
gradient and Hessian-vector product queries.
\end{corollary}

\section{DVR-SOPO}\label{sec:dvr}
\subsection{Algorithm}
Note that the sample complexity~\eqref{eq:bound-drso} of Algorithm~\ref{alg:drsopo} is dominated by the component contributed by the stochastic gradient queries.  In this section, we further improve the complexity of \drsopo{} by variance reduction.
By incorporating the Hessian-aided variance reduction technique (HAVR, \citet{RF1}), we develop a dimension and variance-reduced second-order method for policy optimization, dubbed \dvrsopo{}, which  obtains the optimal $\mathcal{O}(\epsilon^{-3})$ worst-case complexity bound. 

\paragraph{Variance reduction}
We first introduce the  variance reduction technique, which helps to construct a more sample-efficient gradient estimator $g_t$. 
Let $\left\{\theta_s\right\}_{s=0}^t$ be the solution sequence. The gradient  $\nabla J(\theta_t)$ can be written in a path-integral form:
$
\nabla J\left(\theta_t\right)=\nabla J\left(\theta_0\right)+\sum_{s=1}^t \bsbra{\nabla J\left(\theta_s\right)-\nabla J\left(\theta_{s-1}\right)}.
$
Let $\xi_s$ be an unbiased estimator for the gradient difference $\nabla J\left(\theta_s\right)-\nabla J\left(\theta_{s-1}\right)$. We can recursively construct the following estimator for $\nabla J\left(\theta_t\right)$ :
\[
g_t= \begin{cases}g\left(\theta_t ; \mathcal{M}_0\right) & \bmod (t, q)=0, \\ g_{t-1}+\xi_t & \bmod (t, q) \neq 0\end{cases}
\]
where $\mathcal{M}_0$ is a mini-batch of trajectories sampled from $p\left(\cdot \mid \pi_\theta\right)$ and $q$ is a given epoch length. In other words, we directly estimate the gradient using a mini-batch $\mathcal{M}_0$  every $p$ iterations, and we maintain an unbiased estimate by recursively adding the correction term $\xi_t$ to the current estimate $g_{t-1}$ in the other iterations.

\paragraph{Construction of $\xi_t$.} We next describe the key idea of HAVR. Due to the non-oblivious stochastic setting of RL (which means that the distribution of random variable $\tau$ is also influenced by independent variable $\theta$), we can't have access to unbiased samples at $\theta_t$ and $\theta_{t-1}$ simultaneously as classical SVRG does. Therefore, we solve it by utilizing Hessian information: from the Taylor's expansion, the gradient difference can be written as
\begin{equation}\label{eq:diff-estimate}
    \begin{aligned}
\nabla J\left(\theta_t\right)-\nabla J\left(\theta_{t-1}\right) &=\int_0^1\left[\nabla^2 J(\theta(a)) \cdot v\right] d a \\
&=\left[\int_0^1 \nabla^2 J(\theta(a)) d a\right] \cdot v,
\end{aligned}
\end{equation}
where we denote $\theta(a):=a  \theta_t+(1-a) \theta_{t-1}$ and $v:=\theta_t-\theta_{t-1}$. Note that the integral in \eqref{eq:diff-estimate} can be understood as  the expectation $\Expect_a\left[\nabla^2 J(\theta(a))\right]$, where $a$ is a random variable uniformly sampled in between $[0,1]$. Hence, with our unbiased Hessian estimator (\ref{eq:stoc-hessian}),  we can rewrite the gradient difference as  
\begin{equation}\label{eq:grad-diff-expect}
 \begin{aligned}
&\nabla J\left(\theta_t\right)-\nabla J\left(\theta_{t-1}\right)\\ &=\Expect_{a \sim U[0,1],\tau(a) \sim p(\cdot; \pi_{\theta(a)})}[H(\theta(a);\tau(a))] \cdot v
\end{aligned}
\end{equation}
Let $\widehat{\Mcal}_g=\{(a,\tau(a))\}$ be set of samples,  we can construct the estimator $\xi_t$ by
\begin{equation}\label{eq:xi_t}
 \begin{aligned}
\xi_t &=\frac{1}{|\widehat{\Mcal}_g|} \sum_{(a,\tau(a)) \in \widehat{\Mcal}_g} H(\theta(a);\tau(a))\cdot v\\
\end{aligned}
\end{equation}

\paragraph{Comparison with other VR techniques}
    We emphasize that other variance-reduced techniques, such as SRVR-PG,  can also be applied in our methods to improve the estimation of $\nabla J(\theta)$. However, HAVR appears to have some advantages. In order to construct an unbiased estimator of \eqref{eq:xi_t}, one relies on the assumption of expected LS, which is weaker than the point-wise LS often imposed on standard VR methods. Moreover, HAVR does not need to assume the bounded variance of importance sampling, which is also a strong assumption imposed by the standard VR techniques.

Based on the variance reduction technique, we develop the basic \dvrsopo{} algorithm in Algorithm~\ref{alg:dvrsopo}. A more detailed comparison of the sample complexity with other state-of-the-art policy gradient methods in Table~\ref{tab:complexity-compare}. Similar to the development of DR-SOPO, we will present a more practical version of DVR-SOPO in Algorithm~\ref{prac-dvrsopo}.

\begin{algorithm}[h]
	\caption{Basic DVR-SOPO} 
	\label{alg:dvrsopo}
	\begin{algorithmic}[1]%
       \STATE Given $T$, $\Delta$, $q$
		\FOR{$t=1, \dots ,T$}
        \IF{$\mod (t,q)=0$}
	    \STATE Compute $g_t$ based on sampled trajectories  $\Mcal_0$  
        \ELSE
        \STATE  Compute $\xi_t$ based on sampled  trajectories $\widehat{\Mcal}_g$ and update:
        \[g_t=g_{t-1}+\xi_t\]
        \ENDIF
        \STATE Collect samples trajectories $\Mcal_H$  and  construct $H_t$
        \STATE Compute  stepsize $\alpha_1,\alpha_2 $ by solving the DRTR problem \eqref{eq:DRTR}
        \STATE Update: $\theta_{t+1} \gets \theta_{t}-\alpha_1 g_t+\alpha_2 d_t$
        \ENDFOR
	    \end{algorithmic}
\end{algorithm}

\subsection{Complexity Analysis}

DVR-SOPO  relies on the same assumptions of \drsopo{}.  
With the help of HAVR, we can bound the variance of gradient estimator more sample-efficiently. 
\begin{lemma}[Gradient variance bound for DVR-SOPO]\label{var-bound for g_DVR-SOPO}
    Recall the definition of $g_t$ in the Algorithm 2. By setting $\epsilon\leq \frac{G_H^2}{4}$, $q=\frac{1}{8 \epsilon^{1/2}},|\widehat{\Mcal}_g|=\frac{288G_H^2}{M^2 \epsilon^{3/2}}$, and $\left|\mathcal{M}_0\right|=\frac{288G_g^2}{ \epsilon^2}$, we have
$$
\Expect \bsbra{\left\|g_t-\nabla J\left(\theta_t\right)\right\|^2} \leq \frac{\epsilon^2}{144M^2}.
$$
\end{lemma}
\begin{theorem}[Convergence rate of \dvrsopo{}]\label{thm conv-DVR}
     Suppose Assumptions \ref{assm:bdd-reward}-\ref{assm:drsom} hold. Let $\Delta_t=\Delta=\frac{2\sqrt{\epsilon}}{M}$ and $\Delta_J$ be a constant number that s.t. $\Delta_J\geq J(\theta_0)-J^*$. If we set $\epsilon\leq \frac{G_H^2}{4}$, $q=\frac{1}{8 \epsilon^{1/2}},|\widehat{\Mcal}_g|=\frac{288G_H^2}{M^2 \epsilon^{3/2}}$, $\left|\mathcal{M}_0\right|=\frac{288G_g^2}{\epsilon^2}$, $\left|\mathcal{M}_H\right|=\frac{22\times24^2G_H^2\log(d)}{\epsilon}$, and $T=\frac{24M^2\Delta_J}{ \epsilon^{\frac{3}{2}}}$ in Algorithm~\ref{alg:dvrsopo}, then we have
\[
\Expect[\left\|\nabla J\left(\theta_{\bar{t}}\right)\right\|] \leq \frac{(3+4\tilde{C})}{M}\epsilon ,\\
\Expect[\lambda_{\min}(\tilde{\nabla}^2J(\theta_{\bar{t}}))]\geq-3\sqrt{\epsilon},
\]
where $\bar{t}$ is uniformly sampled from $\{1, \ldots, T\}$.
Moreover, with probability at least $\frac{7}{8}$,  we have $$
\left\|\nabla J\left(\theta_{\bar{t}}\right)\right\| \leq \frac{(12+16\tilde{C})}{M}\epsilon.
$$
\end{theorem}

\begin{corollary}[Sample complexity of DVR-SOPO]\label{cor:sample-complexity-dvr}
Under all the assumptions and parameter settings in Theorem~\ref{thm conv-DVR}, Algorithm~\ref{alg:dvrsopo}
  returns a point $\widehat{x}$ such that $\Expect[\|\nabla J(\widehat{x})\|] \leq \frac{(3+4\Tilde{C})}{M}\epsilon$, $\Expect[\lambda_{\min}(\tilde{\nabla}^2 J(\widehat{x})]\geq -3\sqrt{\epsilon} $ after performing at most
$$
\mathcal{O}\left(\frac{\Delta_J (8M^2G_g^2+G_H^2)}{\epsilon^3} +\frac{\Delta_J M^2G_H^{2} }{\epsilon^{2.5}}\right)
$$
gradient and Hessian-vector product queries.
\end{corollary}     
\begin{proof}
    In view of Lemma \ref{var-bound for g_DVR-SOPO} and \ref{lm var for hessian}, for each iteration of \dvrsopo{}, we need $N_1=\frac{288(8G_g^2+\frac{G_H^2}{M^2})}{\epsilon^{3/2}}$ samples for gradient estimator $g_t$ and $N_2=\frac{22\cdot576G_H^2\log(d)}{\epsilon}$ samples for Hessian estimator $H_t$. Following Theorem \ref{thm conv-DVR}, the total number of samples required by \dvrsopo{} is
    \begin{equation*}
    \small
        (N_1+N_2)T=\mathcal{O}\Brbra{\frac{\Delta_J (8M^2G_g^2+G_H^2)}{\epsilon^3} +\frac{\Delta_J M^2G_H^{2}\log(d) }{\epsilon^{2.5}}}.
    \end{equation*}
\end{proof}

\section{Practical Implementation}\label{sec:implement}
To achieve better empirical performance, we describe more practical versions of DR-SOPO and DVR-SOPO that dynamically adjust the trust region radius $\Delta_t$. Consider the reduction ratio for $m_t^\alpha$~\eqref{eq:DRTR} at iterate $\theta_t$ 
\begin{equation}\label{eqDRSOM RATIO}
\rho_t:=\frac{J\left(\theta_t\right)-J\left(\theta_t+d_{t+1}\right)}{m_t(0)-m_t\left(\alpha_t\right)}.
\end{equation}
If $\rho_t$ is too small, our quadratic model is somehow inaccurate, which prompts us
to reduce $\Delta_t$. 

In practice, it is  difficult to adjust properly $\Delta_t$ in \eqref{eq:DRTR}. To resolve this issue, we consider the following radius-free problem:
\begin{equation}\label{radius-free}
\begin{aligned}
\beta_\alpha(\lambda_t)= \min_{\alpha \in \mathbb{R}^2} J\left(\theta_t\right)+c_t^\trans \alpha+\frac{1}{2} \alpha^\trans Q_t \alpha+\lambda_t\|\alpha\|_{G_t}^2
\end{aligned}
\end{equation}
as an alternative to \eqref{eq:DRTR}. Due to the KKT condition, $\Delta_t$ is implicitly defined by $\lambda_t$, and we can properly adjust $\lambda_t$ to solve $\beta_\alpha(\lambda_t)$ and get sensible amount of decrease. This strategy has been proven effective in \citet{zhang2022drsom} and hence is used in our implementation. The details of  practical DR-SOPO and DVR-SOPO are presented in Appendix~\ref{Practical versions of DR-SOPO and DVR-SOPO}.

 A notable portion of the literature has focused on deriving a better choice of $\Psi_h(\tau)$ in order to reduce the variance in estimating the policy gradient $\nabla J(\theta)$. Conventional approaches include actor-critic algorithms~\citep{bhatnagar2007incremental, konda1999actor}, and adding baselines~\citep{wu2018variance}. The Generalized Advantage Estimation~(GAE) proposed by ~\citet{schulman2015high} is one of the best ways to approximate the advantage function. We emphasize that all these refined advantage estimators can be directly incorporated to our method by placing $\Psi_h(\tau)$ correspondingly. Our experiments use GAE$(\gamma , 1)$, then $\Psi_h(\tau)$ becomes $\Psi_h(\tau)=\sum_{i=h}^{H-1} \gamma^i r(s_i, a_i) - b(s_h),$
where $b$ is the linear baseline.

\section{Numerical Results}\label{sec:numerical}
\begin{figure*}[t]
    \centering
    \begin{minipage}{0.49\linewidth}
        \centering
        \includegraphics[width=1.0\linewidth]{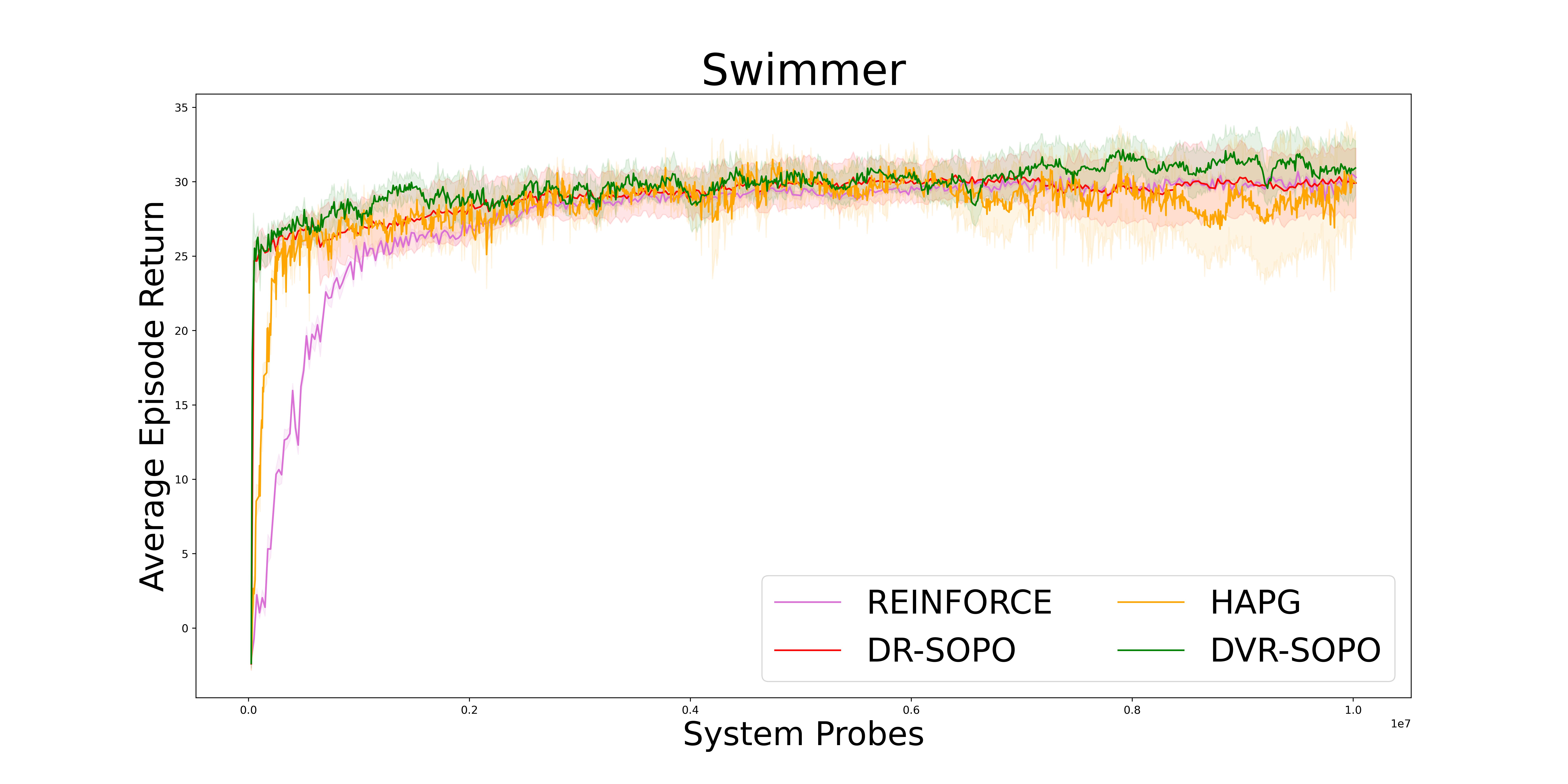}
    \end{minipage}
    \begin{minipage}{0.49\linewidth}
        \centering
        \includegraphics[width=1.0\linewidth]{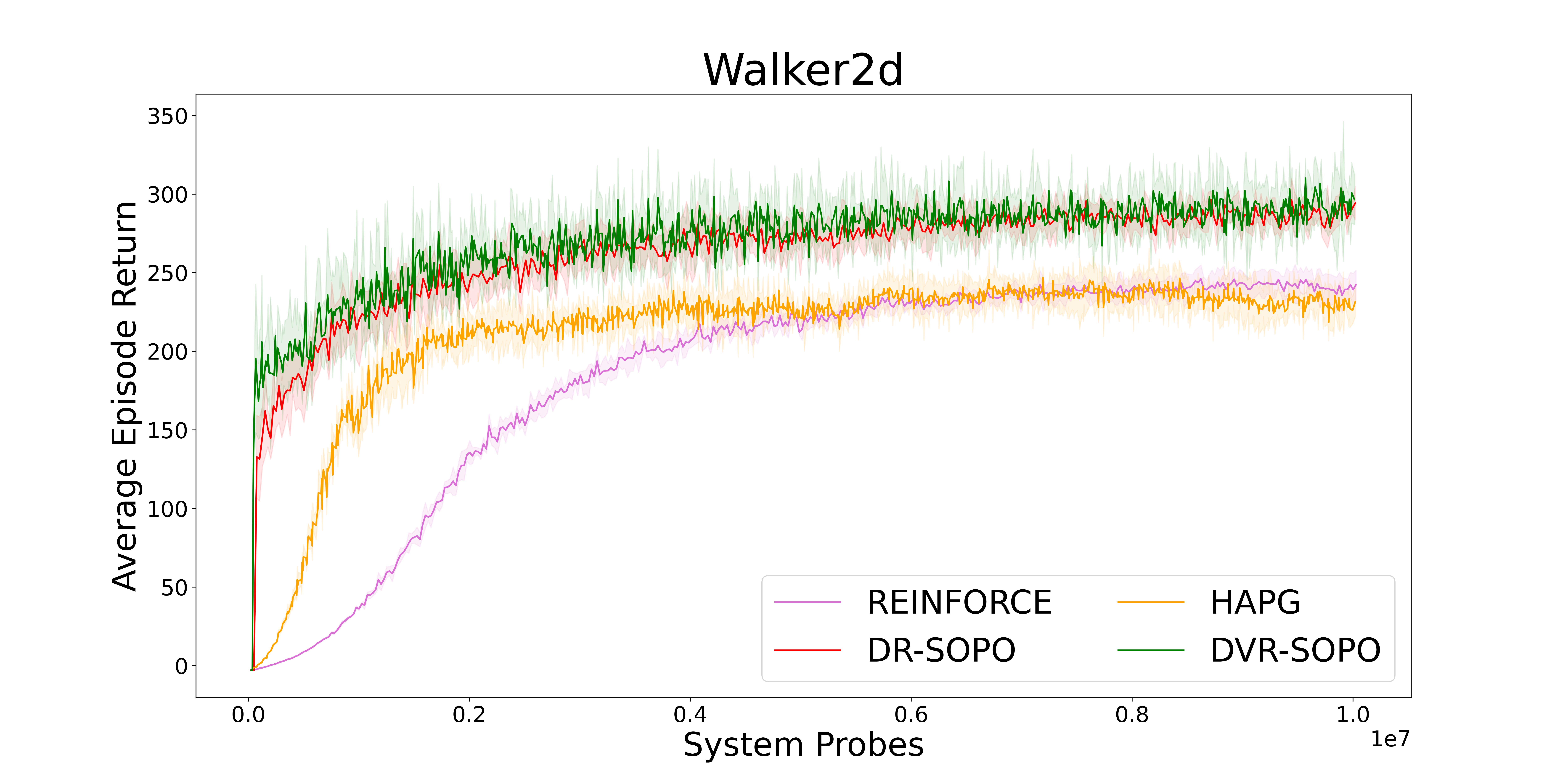}
    \end{minipage}
    
    \begin{minipage}{0.49\linewidth}
        \centering
        \includegraphics[width=1.0\linewidth]{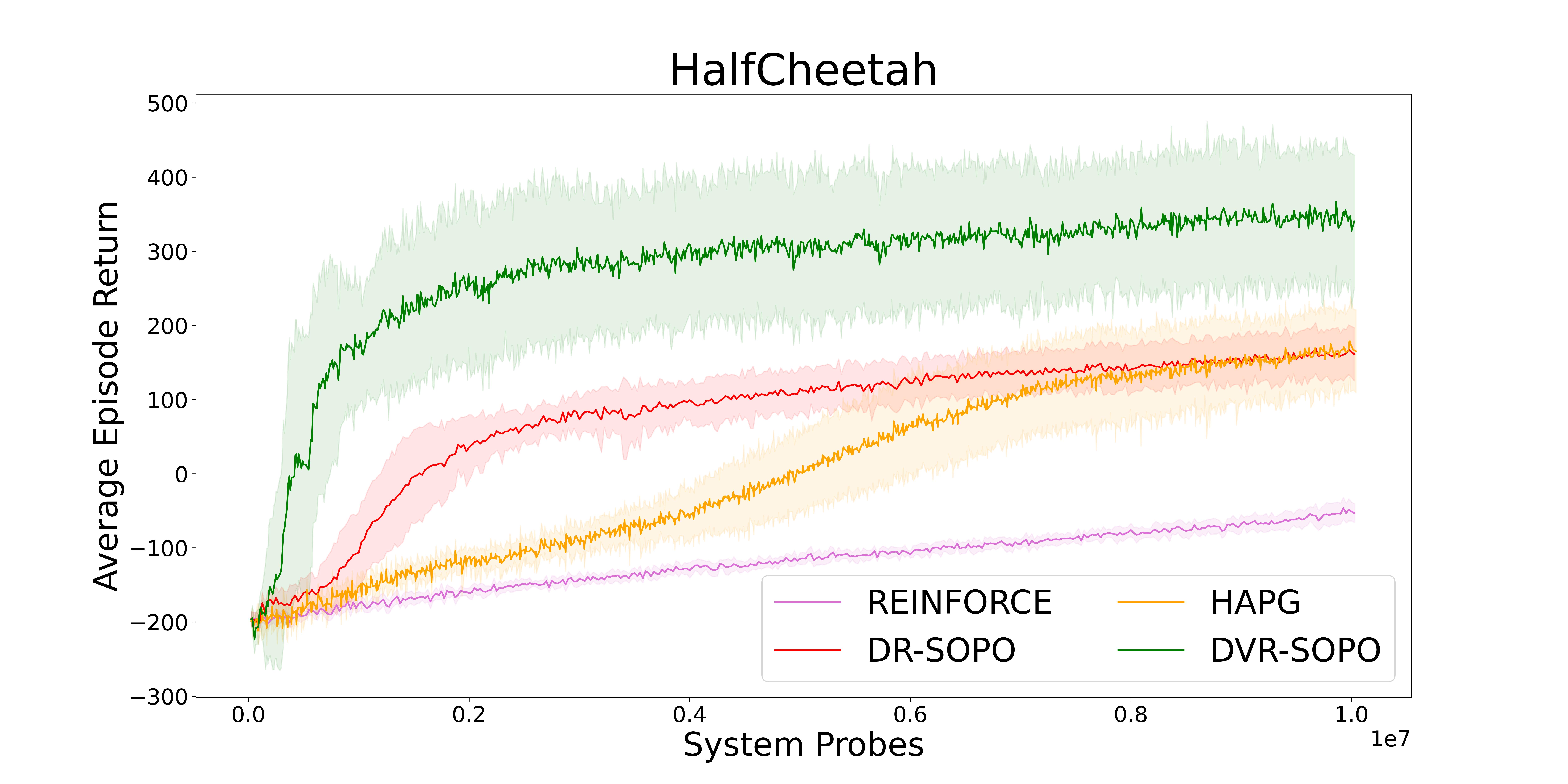}
    \end{minipage}
    \begin{minipage}{0.49\linewidth}
        \centering
        \includegraphics[width=1.0\linewidth]{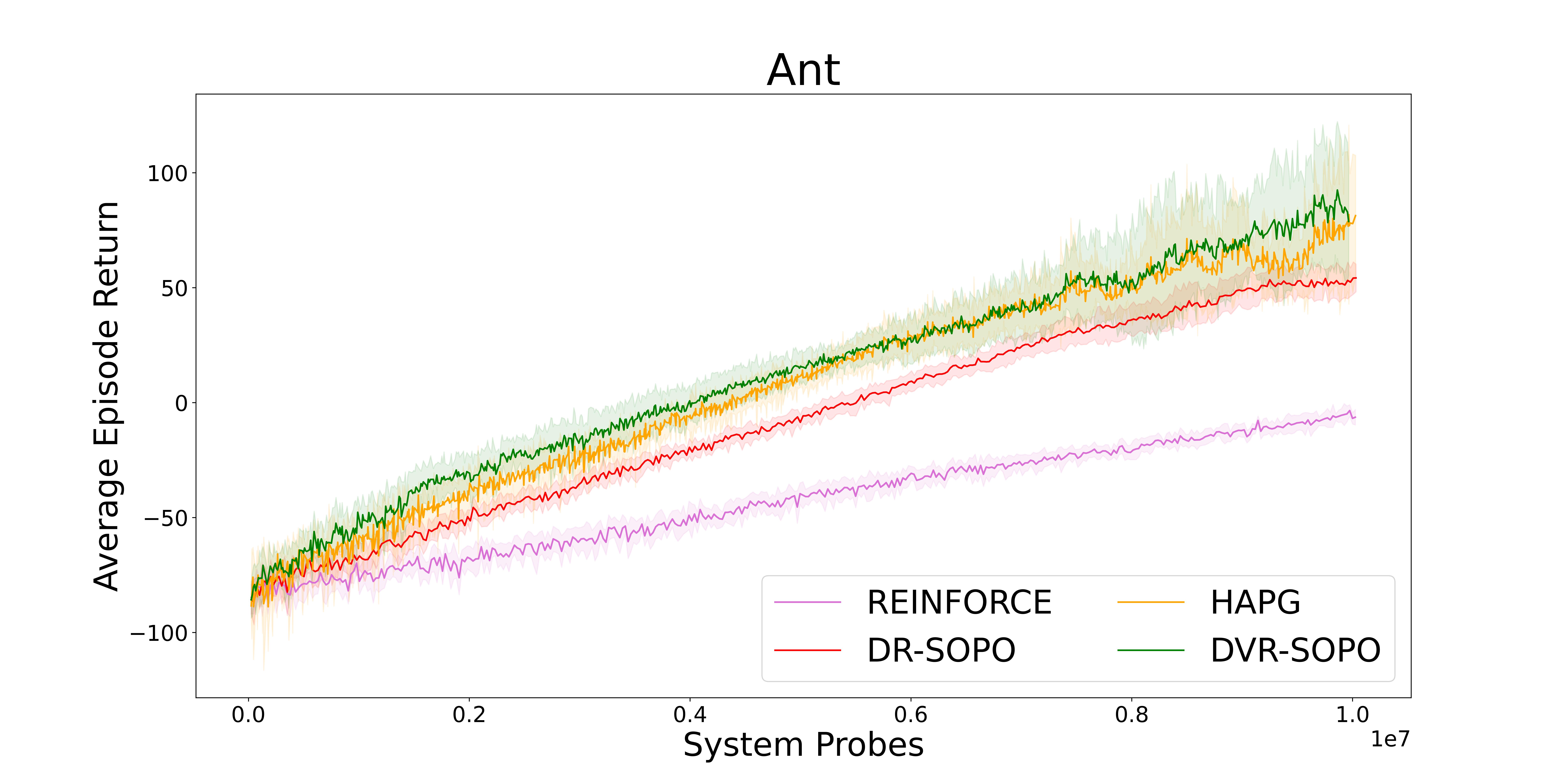}
    \end{minipage}
    \caption{Performance of \drsopo{}, DVR-SOPO and REINFORCE, HAPG on four environments}
    \label{Test Results}
\end{figure*}
In this section, we evaluate the performance of  \drsopo{} and DVR-SOPO and compare them with REINFORCE~\citep{sutton1999policy} and HAPG~\citep{RF1}.  We choose four Mujoco environments~\citep{todorov2012mujoco}: Walker2d-v2, Swimmer-v2, HalfCheetah-v2 and Ant-v2. We implement the four algorithms based on the Garage library \cite{garage} and PyTorch \cite{Paszke_PyTorch_An_Imperative_2019}.

For all the environments, we use deep Gaussian policy, where the mean and variance are parameterized by a fully-connected neural network. We normalize the gradient estimator of all four algorithms for fair comparison. We also initialize all algorithms with the same random policy parameters. 
We repeat the experiment 10 times to reduce the impact of randomness and plot the mean and variance, using the same batch size for all algorithms.  Note that the two variance-reduced algorithms DVR-SOPO and HAPG are double-loop procedures. We set the period and batch size of inner loop to be the same. For the two second-order algorithms \drsopo{} and DVR-SOPO, we set the additional hyper-parameters to be the same. More details of model architecture and hyper-parameter settings are shown in Appendix~\ref{Parameter Settings}.

We use the number of system probes (the number of state transitions) instead of the number of trajectories to measure the sample complexity. System probes is a better criterion since different trajectories might have different number of system probes when a failure flag returned from the environment. To be consistent with previous empirical study, we measure the algorithm performance using the average episode return, which is the negation of the cost/regret.  Hence the higher return indicates better result.

The experiment results are shown in Fig.~\ref{Test Results}. First, we observe that  \drsopo{} consistently outperform REINFORCE in all four environments, and achieving  significantly better final average return in Walker2d, HalfCheetah and Ant.  Moreover,  we observe similar comparison results between DVR-SOPO and HAPG, as \dvrsopo{} converges faster in all four environments and achieves significantly better final average return on Walker2d and HalfCheetah.  These experiment results indeed confirm the advantage of using second-order information. 
Next, we compare the performance of two second-order algorithms \drsopo{} and \dvrsopo{}. In Swimmer, their convergence rates are similar, but DVR-SOPO yields better final average reward. In Walker2d, they yield similar final average reward but DVR-SOPO converge faster than DR-SOPO at the beginning. In HalfCheetah and Ant, DVR-SOPO not only converge faster but yield better final average reward than DR-SOPO. The experiment results suggest that  variance reduction  does improve the performance of  second-order algorithms, which align with our theoretical analysis. 
\section{Discussion}\label{sec:conclusion}
In this paper, we propose several efficient second-order methods, namely \drsopo{} and \dvrsopo{}, for policy optimization. We show that \drsopo{} obtains an $\Ocal(\epsilon^{-3.5})$ complexity for reaching first order stationary condition and a subspace second-order stationary condition.  \dvrsopo{} further improves the rate  to $\Ocal(\epsilon^{-3})$ based on the recently proposed variance reduction technique. 
Our methods have the distinct advantage of only requiring  the computation of gradient and  Hessian-vector product in each iteration, without the need to solve much complicated quadratic subproblems.  As a result, they can be  implemented with efficiency comparable to standard PG methods. 

It is worth noting that while we mainly focus on the dimension-reduced model, it is easy to extend our analysis to develop a "full-dimension'' stochastic trust region algorithm that solves the  standard trust region subproblem instead of the reduced problem~\eqref{eq:DRTR}. Due to the page limits, we give more detail analysis in Appendix section~\ref{sec:FDTR}. 

One interesting direction is to further exploit the connection between our methods and TRPO. It would be interesting to see how to apply our algorithm to TRPO if the objective is replaced by a more accurate quadratic approximation. Additionally, it would be valuable to investigate the performance of our methods for certain parameterized policy when global optimality is guaranteed.

\section*{Acknowledgements}
This work is partially supported by the National Natural Science Foundation of China (NSFC) [Grant NSFC-72150001, 72225009, 11831002]. The authors would like to thank Chuwen Zhang for the discussion on DRSOM and suggestions for the experiments.

\bibliography{ref}
\newpage{}

\appendix

\section{Optimal condition and solution for trust region problems}\label{Optimal condition and solution for trust region problems}
We introduce the widely known optimal conditions for trust region methods.
The vector $\alpha_t$ is the global solution to DRTR problem~\eqref{eq:DRTR} if it is feasible and there exists a Lagrange multiplier $\lambda_t \geq 0$ such that $\left(\alpha_t, \lambda_t\right)$ is the solution to the following equations:
\begin{equation}\label{optimal condi for original}
    \left(Q_t+\lambda G_t\right) \alpha+c_t=0, Q_t+\lambda G_t \succeq 0, \lambda\left(\Delta-\|\alpha\|_{G_t}\right)=0 .
\end{equation}
Then by the optimal condition \eqref{optimal condi for original}, we have the closed form solution of $\alpha_t$:
$$
\alpha_t=-(Q_t+\lambda_tG_t)^{-1}c_t.
$$
Since $\alpha$ only has two dimensions, it can be easily solved numerically. Furthermore, recalling Lemma \ref{lmdrsom1}, by constructing $d_{t+1}=V_{t} \alpha_t$, we can prove that $d_{t+1}$ is the solution to the full-scale problem \eqref{subprob-equiv} such that
\begin{equation}\label{optimal condi for equiv}
    \left(\tilde{H}_{t}+\lambda_{t} I\right) d_{t+1}+G_t=0, \tilde{H}_{t}+\lambda_{t} I \succeq 0, \lambda_{t}\left(\left\|d_{t+1}\right\|-\Delta_t\right)=0,
\end{equation}

where $\tilde{H}_{t}=V_{t} V_{t}{^\trans} H_{t} V_{t} V_{t}{^\trans}$.
\section{Detailed proofs}
\subsection{Proof of Lemma \ref{lm LG-Variance-bound}}\label{vr hessian estimator}
The proof can be found in \citet{yuan2022general}, which is the best result though under a weaker expected-LS regularity. We know that our gradient estimator (PGT estimator, \citep{sutton2000empirical}) is equivalent to GPOMDP estimator \cite{baxter2001infinite}:
$$
g(\theta;\tau)=\sum_{h=0}^{H-1}\left(\sum_{t=0}^h \nabla_{{\theta}} \log \pi_{{\theta}}\left(a_t \mid s_t\right)\right)\left(\gamma^h r\left(s_h, a_h\right)\right)
$$
So we will also use the GPOMDP estimator definition if needed. 
\begin{proof}
$$
\begin{aligned}
\Expect_\tau\left[\|g(\theta;\tau)\|^2\right] & {=} \Expect_\tau\left[\left\|\sum_{t=0}^{H-1} \gamma^{t / 2} r(s_t, a_t) \gamma^{t / 2}\left(\sum_{k=0}^t \nabla_\theta \log \pi_\theta\left(a_k \mid s_k\right)\right)\right\|^2\right] \\
& \leq \Expect_\tau\left[\left(\sum_{t=0}^{H-1} \gamma^t r(s_t, a_t)^2\right)\left(\sum_{k=0}^{H-1} \gamma^k\left\|\sum_{k^{\prime}=0}^k \nabla_\theta \log \pi_\theta\left(a_{k^{\prime}} \mid s_{k^{\prime}}\right)\right\|^2\right)\right] \\
& \stackrel{\clubsuit}{\leq} \frac{R^2}{1-\gamma} \cdot \sum_{k=0}^{H-1} \gamma^k \Expect_\tau\left[\left\|\sum_{k^{\prime}=0}^k \nabla_\theta \log \pi_\theta\left(a_{k^{\prime}} \mid s_{k^{\prime}}\right)\right\|^2\right] \\
& \stackrel{(\Diamond)}{=} \frac{R^2}{1-\gamma} \cdot \sum_{k=0}^{H-1} \gamma^k \sum_{k^{\prime}=0}^k \Expect_\tau\left[\left\|\nabla_\theta \log \pi_\theta\left(a_{k^{\prime}} \mid s_{k^{\prime}}\right)\right\|^2\right] \\
& \leq \frac{G^2 R^2}{1-\gamma} \cdot \sum_{k=0}^{H-1} \gamma^k(k+1) \\
& \leq \frac{G^2 R^2}{(1-\gamma)^3}:=G_g^2,
\end{aligned}
$$
where $(\clubsuit)$ uses the Assumption~\ref{assm:bdd-reward}, and  $(\Diamond)$ is due to the fact that for any $h \neq h^{\prime}$:
\begin{equation}\label{eq RL trick}
    \Expect_\tau\left[\nabla_\theta \log \pi_\theta\left(a_h \mid s_h\right)){^\trans} (\nabla_\theta \log \pi_\theta\left(a_{h^{\prime}} \mid s_{h^{\prime}}\right))\right]=0.
\end{equation}
To derive the variance-reduced Hessian estimator and bound its variance, note that 
$$
\begin{aligned}
\Expect_\tau[\|H(\theta;\tau)\|^2] & \leq 2\Expect[\| \nabla_{{\theta}} g(\theta;\tau)\|^2] +2\Expect_\tau[\| g(\theta;\tau) \nabla_{{\theta}} \log p(\tau \mid {\theta})^\trans\|^2].
\end{aligned}
$$
Plugging the GPOMDP definition of $g(\theta;\tau)$ in this term yields
\begin{equation*}
    \begin{aligned}
    \Expect_\tau\left[\|\nabla g(\theta;\tau)\|^2\right]
    &\leq \Expect_\tau  \left[(\sum_{h=0}^{H-1}\gamma^h R^2)\left( \sum_{h=0}^{H-1} \gamma^h\left\|\sum_{t^{\prime}=0}^h\nabla_{{\theta}}^2 \log \pi_{{\theta}}\left(a_{t^{\prime}} \mid s_{t^{\prime}}\right)\right\|^2\right)\right]\\
    &\leq \frac{R^2}{1-\gamma}\cdot \sum_{h=0}^{H-1} \gamma^h \Expect_\tau\left[\left\|\sum_{t^{\prime}=0}^h\nabla_{{\theta}}^2 \log \pi_{{\theta}}\left(a_{t^{\prime}} \mid s_{t^{\prime}}\right)\right\|^2\right]\\
    & \leq \frac{R^2}{1-\gamma}\cdot\sum_{h=0}^{H-1} \gamma^h(h+1)\sum_{t^{\prime}=0}^h\Expect_\tau\left[\left\|\nabla_{{\theta}}^2 \log \pi_{{\theta}}\left(a_{t^{\prime}} \mid s_{t^{\prime}}\right)\right\|^2\right]\\
    &\leq \frac{L^2 R^2}{1-\gamma}\cdot \sum_{h=0}^{H-1} \gamma^h(h+1)^2\\
& \leq \frac{ 2L^2 R^2}{(1-\gamma)^4}.
\end{aligned}
\end{equation*}
For the ease in notation, let us denote $z_t=\nabla_{{\theta}} \log \pi_{{\theta}}\left(a_t \mid s_t\right)$. It follows that
\begin{equation*}
    \begin{aligned}
        & \Expect_\tau \left[\left\| g(\theta;\tau) (\nabla_{{\theta}} \log p(\tau \mid {\theta})){^\trans}\right\|^2\right] \\
& =\Expect_\tau \left[\left\|\sum_{h=0}^{H-1}\gamma^h r\left(s_h, a_h\right)\left(\sum_{t=0}^h \nabla_{{\theta}} \log \pi_{{\theta}}\left(a_t \mid s_t\right)\right) \left(\sum_{t^{\prime}=0}^{H-1} (\nabla_{{\theta}} \log \pi_{{\theta}}\left(a_{t^{\prime}} \mid s_{t^{\prime}}\right))^\trans\right)\right\|^2 \right]\\
&\stackrel{\Diamond}{\leq} \Expect_\tau  \left[\left(\sum_{h=0}^{H-1}\gamma^h R^2\right)\left( \sum_{h=0}^{H-1} \gamma^h\Big\|\sum_{t=0}^h z_t\sum_{t^{\prime}=0}^{H-1}z_t^\trans
\Big\|^2\right)\right]\\
& \leq \Expect_\tau  \left[\left(\sum_{h=0}^{H-1}\gamma^h R^2\right)\left( \sum_{h=0}^{H-1} \gamma^h\Big(\sum_{t=0}^h \|z_t\|\sum_{t^{\prime}=0}^{H-1}\|z_{t^{\prime}}\|^\trans
\Big)^2\right)\right]\\
& \leq \frac{R^2}{1-\gamma}\Expect_\tau\left[ \sum_{h=0}^{H-1} \gamma^h\Big(\sum_{t=0}^{H-1} \|z_t\|\sum_{t^{\prime}=0}^{H-1}\|z_{t^{\prime}}\|^\trans
\Big)^2\right]\\
&\stackrel{\sharp}{\leq} \frac{R^2}{1-\gamma}\sum_{h=0}^{H-1} \gamma^hH^4 \Expect_\tau[\|z_t\|^4]\\
&\leq \frac{G^4 R^2}{1-\gamma}\sum_{h=0}^{H-1} \gamma^hH^4\leq \frac{G^4H^4R^2}{(1-\gamma)^2}
    \end{aligned}
\end{equation*}
where $(\Diamond)$ is due to Cauchy-Schwarz inequality, and $(\sharp)$ is  due to the fact that for any  matrix (or scalar) $A_0,A_1,\dots A_n$ with dimension $d_1\cdot d_2$, we
have 
\begin{equation}\label{eq matrix help inequality}
    \begin{aligned}
        \Expect[\|(\sum_{i=0}^nA_i\sum_{i=0}^nA_i^\trans )\|^2]
        & =\Expect[\|\sum_{i=0}^n A_iA_i^\trans+2\sum_{i\neq j}A_iA_j^\trans\|^2 ]\\
        & \leq \Expect[(\sum_{i=0}^n\|A_i\|^2+2\sum_{i\neq j}\|A_i\|\|A_j\|)^2]\\
        & \leq \Expect[\Big((n+1)\sum_{i=0}^n\|A_i\|^2\Big)^2]\\
        &=(n+1)^2\cdot \Expect[(\sum_{i=0}^n\|A_i\|^2)^2]=(n+1)^3\cdot\Expect[\sum_{i=0}^n\|A_i\|^4]. \\
    \end{aligned}
\end{equation}
Hence, we have 
$$
\Expect_\tau\|H(\theta ;\tau)\|^2\leq 2 \Expect_\tau \left[\left\| g(\theta;\tau) (\nabla_{{\theta}} \log p(\tau \mid {\theta})){^\trans}\right\|^2\right]+2  \Expect_\tau\left[\|\nabla g(\theta;\tau)\|^2\right] \leq \frac{2H^4G^4 R^2(1-\gamma)^2 +4L^2R^2}{(1-\gamma)^4}:=G^2_H.
$$
Finally, using $\Expect_\tau[(X-\Expect_\tau[X])^2]\leq\Expect_\tau[X^2]$ for all random variable X, we have
\begin{equation*}
\begin{aligned}
\Expect_\tau\|g(\theta ;\tau)-\nabla J(\theta)\|^2 & \leq \Expect_\tau\|g(\theta ;\tau)\|^2\leq G_g^2, \\
\Expect_\tau\|H(\theta ;\tau)-\nabla^2 J(\theta)\|^2 & \leq \Expect_\tau\|H(\theta ;\tau)\|^2\leq G_H^2.
\end{aligned}
\end{equation*}
\end{proof}
\subsection{More discussion about Hessian estimator}\label{discuss with H estimator}
\textbf{1. A Variance-reduced unbiased estimator}: We claim here that by constructing a variance reduced Hessian estimator, the variance of $H(\theta;\tau)$ can actually by bounded more tightly, which is without parameter H. To do so, note that
$$
\begin{aligned}
& \Expect_\tau\left[g(\theta;\tau) (\nabla_{{\theta}} \log p(\tau \mid {\theta})){^\trans}\right] \\
& =\Expect_\tau \left[\sum_{h=0}^{H-1}\left(\sum_{t=0}^h \nabla_{{\theta}} \log \pi_{{\theta}}\left(a_t \mid s_t\right)\right) \gamma^h r\left(s_h, a_h\right)( \nabla_{{\theta}} \log p(\tau \mid {\theta})){^\trans}\right] \\
& \stackrel{\diamondsuit}{=}\Expect_\tau \left[\sum_{h=0}^{H-1}\left(\sum_{t=0}^h \nabla_{{\theta}} \log \pi_{{\theta}}\left(a_t \mid s_t\right)\right) \gamma^h r\left(s_h, a_h\right)\left(\sum_{t^{\prime}=0}^{H-1} (\nabla_{{\theta}} \log \pi_{{\theta}}\left(a_{t^{\prime}} \mid s_{t^{\prime}}\right))^\trans\right) \right]\\
& \stackrel{\heartsuit}{=} \Expect_\tau \left[\sum_{h=0}^{H-1}u_h\cdot u_h^\trans \gamma^h r\left(s_h, a_h\right)]\right],
\end{aligned}
$$
where $u_h= \sum_{t=0}^h \nabla_{{\theta}} \log \pi_{{\theta}}\left(a_t \mid s_t\right)$, $(\diamondsuit)$ is due to $\nabla_{{\theta}} \Pcal\left(s_{t^{\prime}+1} \mid s_{t^{\prime}}, a_{t^{\prime}}\right)=0$ and $(\heartsuit)$ uses property \eqref{eq RL trick}.

Let $H^{\prime}(\theta;\tau)=\sum_{h=0}^{H-1} \gamma^h r\left(s_h, a_h\right) u_h\cdot u_h^\trans +\nabla g(\theta;\tau)$, it is also an unbiased estimator of $\nabla^2 J(\theta)$. Moreover, we have 
\begin{equation*}
   \begin{aligned}
    &\Expect_\tau\left[\Big\|\sum_{h=0}^{H-1} \gamma^h r\left(s_h, a_h\right) ( u_h\cdot u_h^\trans)\Big\|^2 \right] \\
    &\stackrel{\Diamond}{\leq} \Expect_\tau  \left[\left(\sum_{h=0}^{H-1}\gamma^h R^2\right)\left( \sum_{h=0}^{H-1} \gamma^h\Big\|u_h\cdot u_h^\trans\Big\|^2\right)\right]\\
    &\leq \frac{R^2}{1-\gamma}\cdot \sum_{h=0}^{H-1} \gamma^h \Expect_\tau\left[\Big\|u_h\cdot u_h^\trans\Big\|^2\right]\\
    & \stackrel{\sharp}{\leq} \frac{R^2}{1-\gamma}\cdot \sum_{h=0}^{H-1} \gamma^h(h+1)^3\sum_{t=0}^h\Expect_\tau\left[\big\|\nabla_{{\theta}} \log \pi_{{\theta}}\left(a_t \mid s_t\right)\big\|^2\right]\\
    &\leq \frac{G^4 R^2}{1-\gamma}\cdot \sum_{h=0}^{H-1} \gamma^h(h+1)^4\\
& \leq \frac{ 24G^4 R^2}{(1-\gamma)^6},
\end{aligned}
\end{equation*}
where $(\Diamond)$ is due to Cauchy-Schwarz inequality, and $(\sharp)$ is  due to \eqref{eq matrix help inequality}.

Hence, we have 
$$
\Expect_\tau\|H^{\prime}(\theta ;\tau)\|^2\leq 2 \Expect_\tau\left[\Big\|\sum_{h=0}^{H-1} \gamma^h r\left(s_h, a_h\right)u_tu_t^\trans
 \Big\|^2 \right]+2  \Expect_\tau\left[\|\nabla g(\theta;\tau)\|^2\right] \leq \frac{48G^4 R^2+4L^2R^2(1-\gamma)^2 }{(1-\gamma)^6}:={G_H^{\prime}}^2
$$
In practice, this method requires parallelism and a significant amount of computing resources as it needs to be backpropagated multiple times ($H$ times).

\textbf{2. Biased Hessian Estimation}: 

Consider the biased estimator of $\nabla^2 J(\theta)$:
\begin{equation}\label{Biased Hessian Estimator}
    H_\mu(\theta ; \tau)= \nabla g(\theta;\tau)+\mu g(\theta;\tau) \nabla \log p(\tau; {\theta})^\trans
\end{equation}
Note that when $\mu = 1$, this bound becomes the unbiased Hessian estimator. However, a biased Hessian estimator which has better practical variance property. We can further derive its MSE bound:
\begin{equation*}
\begin{aligned}
& \Expect_\tau[\|H_{\mu}(\theta; \tau)-\nabla^2 J(\theta)\|^2]\\
&= \Var (H_{\mu}(\theta; \tau)) + \|\Expect_\tau[H_{\mu}(\theta; \tau)] - \nabla^2 J(\theta)\|^2 \\
&=\Var (\nabla g(\theta;\tau)+\mu g(\theta;\tau) \nabla \log p(\tau; {\theta})^\trans) + \|\Expect_\tau\sbra{(\mu - 1) g(\theta;\tau) \nabla \log p(\tau; {\theta})^\trans}\|^2 \\
& = (\mu - 1)^2 \|\Expect_\tau\sbra{g(\theta;\tau) \nabla \log p(\tau; {\theta})^\trans}\|^2 + \Var (\nabla g(\theta;\tau)) + \mu^2 \Var \left[ g(\theta;\tau) (\nabla_{{\theta}} \log p(\tau \mid {\theta})){^\trans}\right] \\ 
& \qquad + 2\mu \Expect_\tau\left[\left[\nabla g(\theta;\tau) - \Expect_\tau\left[\nabla g(\theta;\tau)\right]\right]\left[g(\theta;\tau) \nabla \log p(\tau; {\theta})^\trans - \Expect_\tau\left[g(\theta;\tau) \nabla \log p(\tau; {\theta})^\trans\right]\right]\right] \\
& = (1-2\mu ) \|\Expect_\tau\sbra{g(\theta;\tau) \nabla \log p(\tau; {\theta})^\trans}\|^2 + \mu^2\Expect_\tau\sbra{\|g(\theta;\tau) \nabla \log p(\tau; {\theta})^\trans\|^2} + \Var (\nabla g(\theta;\tau))\\ 
& \qquad + 2\mu \Expect_\tau\left[\left[\nabla g(\theta;\tau) - \Expect_\tau\left[\nabla g(\theta;\tau)\right]\right]\left[g(\theta;\tau) \nabla \log p(\tau; {\theta})^\trans - \Expect_\tau\left[g(\theta;\tau) \nabla \log p(\tau; {\theta})^\trans\right]\right]\right]\\
& \leq 2\mu \sqrt{ \Expect_\tau\left[\|\nabla g(\theta;\tau) - \Expect_\tau\left[\nabla g(\theta;\tau)\right]\|^2\right] \Expect_\tau\left[\|g(\theta;\tau) \nabla \log p(\tau; {\theta})^\trans - \Expect_\tau\left[g(\theta;\tau) \nabla \log p(\tau; {\theta})^\trans\right]\|^2\right]}\\
& \qquad + (\mu-1)^2\Expect_\tau\sbra{\|g(\theta;\tau) \nabla \log p(\tau; {\theta})^\trans\|^2} + \Var (\nabla g(\theta;\tau))\\ 
& \leq \Expect_\tau\left[\|\nabla g(\theta;\tau)\|^2\right] + (\mu - 1)^2 \Expect_\tau \left[\left\| g(\theta;\tau) (\nabla_{{\theta}} \log p(\tau \mid {\theta})){^\trans}\right\|^2\right] \\
&\qquad + 2\mu \sqrt{\Expect_\tau\left[\|\nabla g(\theta;\tau)\|^2\right] \Expect_\tau \left[\left\| g(\theta;\tau) (\nabla_{{\theta}} \log p(\tau \mid {\theta})){^\trans}\right\|^2\right]}\\
& \leq \frac{ 2L^2 R^2}{(1-\gamma)^4} + (\mu - 1)^2 \frac{G^4H^4R^2}{(1-\gamma)^2} + 2\mu \sqrt{\frac{ 2L^2 R^2}{(1-\gamma)^4} \frac{G^4H^4R^2}{(1-\gamma)^2}}.
\end{aligned}
\end{equation*}
 By setting the value $ \mu = 1 - \frac{\sqrt{2}L}{(1-\gamma)G^2H^2}$, 
 we obtain the tightest bound:
$$
\frac{2\sqrt{2}LR^2G^2H^2}{(1-\gamma)^3}.
$$
 In experiment, we find that using a biased estimator often leads to better performance, and hence adjust $\mu$ as a hyper-parameter.
\subsection{Proof of Lemma \ref{lm sgd_var-bound for g-estimator}}
\begin{proof}
     \begin{equation*}
    \begin{aligned}	\Expect_t\left[\left\|{g}_t-\nabla J\left(\theta_{t}\right)\right\|^2\right]
            & =\Expect_t\Big[\Big\|\frac{1}{|\mathcal{M}_g|}\sum_{\tau \in \mathcal{M}_g}g(\theta,\tau)-\nabla J(\theta_{t})\Big\|^2\Big]\\
            & =\frac{\Expect_t\bsbra{\|\sum_{\tau \in \mathcal{M}_g}g(\theta,\tau)-|\mathcal{M}_g|\cdot\nabla^2 J(\theta_{t})\|}}{|\mathcal{M}_g|^2}\\
            &=\frac{\Expect_t\bsbra{\|g(\theta,\tau)-\nabla J(\theta_{t})\|^2}}{|\mathcal{M}_g|}\\
            & \leq \frac{\epsilon^2}{144 M^2},
    \end{aligned}
\end{equation*}
where the last equality is due to the fact that samples are independent and  $M_g=\frac{144G_g^2}{\epsilon^2}$.

 As for the variance bound of Hessian, 
it is a direct result of the following auxiliary lemma (whose proof can be found in \citet{arjevani2020second}. We derive it here just for completeness) by setting $A_{i}=H(\theta;\tau),B=\nabla^2 J(\theta)$.

\begin{lemma} Let $\left(A_i\right)_{i=1}^m$ be a collection of i.i.d. matrices in $\mathbb{S}^n$, with $\Expect\left[A_i\right]=B$ and $\Expect\left\|A_i-B\right\|^2 \leq$ $\sigma^2$. Then it holds that
$$
\Expect\left\|\frac{1}{m} \sum_{i=1}^m A_i-B\right\|^2 \leq \frac{22 \sigma^2 \log n}{m}.
$$
\end{lemma}
\textbf{Proof:} We drop the normalization by $m$ throughout this proof. We first symmetrize. Observe that by Jensen's inequality we have
$$
\begin{aligned}
\Expect\left\|\sum_{i=1}^m A_i-B\right\|^2 & \leq \Expect_A \Expect_{A^{\prime}}\left\|\sum_{i=1}^m A_i-A_i^{\prime}\right\|^2 \\
& =\Expect_A \Expect_{A^{\prime}}\left\|\sum_{i=1}^m\left(A_i-B\right)-\left(A_i^{\prime}-B\right)\right\|^2\\
&
=\Expect_A \Expect_A \Expect_{\epsilon}\left\|\sum_{i=1}^m \epsilon_i\left(\left(A_i-B\right)-\left(A_i^{\prime}-B\right)\right)\right\|^2 \leq 4 \Expect_A \Expect_\epsilon\left\|\sum_{i=1}^m \epsilon_i\left(A_i-B\right)\right\|^2,
\end{aligned}
$$
where $\left(A^{\prime}\right)_{i=1}^m$ is a sequence of independent copies of $\left(A_i\right)_{i=1}^m$ and $\left(\epsilon_i\right)_{i=1}^m$ are Rademacher random variables. Henceforth we condition on $A$. Let $p=\log n$, and let $\|\cdot\|_{S_p}$ denote the Schatten $p$-norm. In what follows, we will use that for any matrix $X,\|X\| \leq\|X\|_{S_{2 p}} \leq e^{1 / 2}\|X\|$. To begin, we have
$$
\Expect_{\epsilon}\left\|\sum_{i=1}^m \epsilon_i\left(A_i-B\right)\right\|^2 \leq \Expect_{\epsilon}\left\|\sum_{i=1}^m \epsilon_i\left(A_i-B\right)\right\|_{S_{2 p}}^2 \leq\left(\Expect_{\epsilon}\left\|\sum_{i=1}^m \epsilon_i\left(A_i-B\right)\right\|_{S_{2 p}}^{2 p}\right)^{1 / p},
$$
where the second inequality follows by Jensen. We now apply the matrix Khintchine inequality (\cite{mackey2014matrix}, Corollary 7.4), which implies that
$$
\begin{aligned}
\left(\Expect_{\epsilon}\left\|\sum_{i=1}^m \epsilon_i\left(A_i-B\right)\right\|_{S_{2 p}}^{2p}\right)^{1/p} \leq(2 p-1)\left\|\sum_{i=1}^{m}\left(A_i-B\right)^2\right\|_{S_{2 p}} & \leq(2 p-1) \sum_{i=1}^m\left\|\left(A_i-B\right)\right\|_{S_{2 p}}^2 \\
& \leq e(2 p-1) \sum_{i=1}^m\left\|\left(A_i-B\right)\right\|^2 .
\end{aligned}
$$
Putting all the developments so far together and taking expectation with respect to $A$, we have
$$
\Expect\left\|\sum_{i=1}^m A_i-B\right\|^2 \leq 4 e(2 p-1) \sum_{i=1}^m \Expect_{A_i}\left\|\left(A_i-B\right)\right\|^2 \leq 4 e(2 p-1) m \sigma^2 .
$$
To obtain the final result we normalize by $m^2$
\end{proof}
\subsection{Proof of Lemma \ref{lm3.5}}
Recalling $\tilde{\nabla^2} J(\theta_t)=V_t V_t{^\trans}\nabla^2 J(\theta_t)V_t V_t{^\trans} $, we have
\begin{proof}
    \begin{equation*}
    \begin{aligned}
        & \Expect_t||(H_t-\tilde{H}_t)d_{t+1}|| \\
        & \leq \Expect_t[||(\nabla^2 J(\theta_t)-\tilde{\nabla^2} J(\theta_t))d_{t+1}||+ ||(H_t-\nabla^2 J(\theta_t))d_{t+1}||+ ||V_t V_t{^\trans}(\nabla^2 J(\theta_t)-H_t)V_t V_t{^\trans} d_{t+1}||]\\
        & \stackrel{(\Diamond)}{\leq} MC\Delta^2+ \Expect_t[||(H_t-\nabla^2 J(\theta_t))d_{t+1}||]+\Expect_t[\|V_t V_t{^\trans}\|\cdot\|(\nabla^2 J(\theta_t)-H_t)d_{t+1}||]\\
        & \stackrel{(\sharp)}{\leq} MC\Delta^2+2\Expect_t[\|(\nabla^2 J(\theta_t)-H_t)d_{t+1}\|]\\
        & \stackrel{(\natural)}{\leq} MC\Delta^2+2\cdot \frac{\sqrt{\epsilon}}{24}\Delta= M\tilde{C}\Delta^2,
    \end{aligned}
\end{equation*}
where $(\Diamond)$ is due to Assumption \ref{assm:drsom} and the fact that $V_t V_t{^\trans} d_{t+1}=d_{t+1}$; $(\sharp)$  is due to the fact that $V_t{^\trans}V_t=I$ and $V_tV_t{^\trans}$ has the same non-zero eigenvalue with $V_t{^\trans}V_t$ ; $(\natural)$ is due to  Lemma \ref{lm sgd_var-bound for g-estimator} and  the fact that $d_{t+1}\leq \Delta_{t}=\Delta =\frac{2\sqrt{\epsilon}}{M}$. Also, $\tilde{C}=C+\frac{1}{24}$.
\end{proof}
\subsection{Proof of Theorem \ref{thm sgd}}
\begin{proof}

\begin{equation}\label{eq3.1}
        \begin{aligned}
        & J(\theta_{t+1})-J(\theta_{t})\\
                &\leq \nabla J(\theta_{t}){^\trans}d_{t+1}+\frac{1}{2}d_{t+1}^{\trans}\nabla^{2}J(\theta_{t})d_{t+1}+\frac{M}{6}||d_{t+1}||^{3}\\
                &= g_t^\trans d_{t+1}+\frac{1}{2}(d_{t+1})^{\trans}H_td_{t+1}+(\nabla J(\theta_{t})-g_{t})^\trans d_{t+1}+\frac{1}{2}(d_{t+1})^{\trans}(\nabla^{2}J(\theta_{t})-H_t)d_{t+1}+\frac{M}{6}||d_{t+1}||^{3}\\
            &\stackrel{(\Diamond)}{\leq}-\frac{1}{2}\lambda_t \|d_{t+1}\|^2+||\nabla J(\theta_{t})-g_{t}||\Delta_{t}+\frac{1}{2}\|(\nabla^2 J(\theta_{t})-H_{t})\|\Delta_t^2+\frac{M}{6}||d_{t+1}||^{3},
        \end{aligned}
    \end{equation}
  where $(\Diamond)$ is due to Cauchy-Schwarz inequality and Lemma \ref{lm drsom2}. At $t$-th iteration, if $\|d_{t+1}\|=\Delta_t=\frac{2\sqrt{\epsilon}}{M}$, then
   we can bound $\|\lambda_t d_{t+1}\|$ by
  \begin{equation}\label{eq3.2}
    \begin{aligned}
         \|\lambda_td_{t+1}\|&\leq\frac{M}{\sqrt{\epsilon}}(J(\theta_{t})-J(\theta_{t+1}))+\|\nabla J(\theta_t)-g_{t}\|+
                    \frac{1}{2}\|\nabla^2 J(\theta_{t})-H_t\|\Delta_t+\frac{M}{6}||d_{t+1}||^{2},\\
    \end{aligned}
  \end{equation}
 Otherwise, if $d_{t+1}< \Delta_t$ , then by the optimal condition of \eqref{eq:DRTR}, $\lambda_t=0$, so the upper bound \eqref{eq3.2} still holds.
 Therefore, we have
   \begin{equation*}
            \begin{aligned}
               ||\nabla J(\theta_{t+1})||
               & \stackrel{(\natural)}{\leq} || \nabla J(\theta_{t+1})-\nabla J(\theta_{t})-\nabla^2 J(\theta_{t})d_{t+1}||+||\nabla J(\theta_{t})-g_t||\\
               &+||(\nabla^2 J(\theta_{t})-H_t)d_{t+1}||+||g_t+\widetilde{H}_kd_{t+1}||+||(\widetilde{H}_k-H_t)d_{t+1}|| \\
               & \stackrel{(\Diamond)}{\leq} \frac{M\Delta^2}{2} +||\nabla J(\theta_{t})-g_t||+\Delta||\nabla^2 J(\theta_{t})-H_t||+\|g_t+\widetilde{H}_kd_{t+1}\|+M\tilde{C}\Delta^2\\
               & \stackrel{(\sharp)}{\leq}\frac{M\Delta^2}{2} +||\nabla J(\theta_{t})-g_t||+\Delta||\nabla^2 J(\theta_{t})-H_t||+\|\lambda_td_{t+1}\|+M\tilde{C}\Delta^2,
            \end{aligned}
              \end{equation*}
    where ($\natural$) is due to the triangle inequality, ($\Diamond$) is due to Taylor approximation, Cauchy-Schwarz inequality and Lemma \ref{lm3.5}, ($\sharp$) is due to the optimal condition of DRTR \eqref{optimal condi for equiv}.
    Plugging the upper bound of $\|\lambda_t d_{t+1}\|$ into the above equation, summing from $t = 0$ to
$T-1$  and then taking the expectation on both sides, we obtain
\begin{equation*}
     \begin{aligned}
                    &\frac{1}{T}\sum_{t=0}^{T-1}\Expect[\|\nabla J(\theta_{t+1})\|]\\
                    &\leq \frac{M}{\sqrt{\epsilon}T}\Expect\left[\sum_{t=0}^{T-1}(J(\theta_t)-J(\theta_{t+1}))\right]+\frac{2}{T}\sum_{t=0}^{T-1}\Expect[||\nabla J(\theta_{t})-g_t||]+\frac{3}{2T}\sum_{t=0}^{T-1}\Delta \Expect[||\nabla^2 J(\theta_{t})-H_t||]+\frac{2+3\tilde{C}}{3}M\Delta^2\\
                   &\leq \frac{M}{\sqrt{\epsilon}T}(J(\theta_0)-J^*)+\frac{(71+96\tilde{C})}{24M}\epsilon,
                \end{aligned}
\end{equation*}
where the last inequality is due to Lemma \ref{lm sgd_var-bound for g-estimator} and $\Delta=\frac{2\sqrt{\epsilon}}{M}$. Then we have the desired result    by taking $T=\frac{24 M^2\Delta_J}{ \epsilon^{\frac{3}{2}}}$ and $\bar{t}$ be uniformly sampled from $0, \dots, T-1$.

Moreover, by  Markov inequality ($\Prob(X\geq a)\leq \frac{\Expect[X]}{a}$ for any non-negative random variable $X$), with probability at least $\frac{7}{8}$, we have
$$
\left\|\nabla J\left(\theta_{\bar{t}+1}\right)\right\| \leq \frac{(12+16\tilde{C})}{M}\epsilon.
$$
To derive the second-order condition, from \eqref{eq3.1} and the optimal condition \eqref{optimal condi for original} of DRTR, we have 
\begin{equation}\label{eq3.3}
\begin{aligned}
			& \Expect[J(\theta_{t+1})-J(\theta_{t})]		\\
   &\leq \Expect\left[-\frac{1}{2}\lambda_t \|d_{t+1}\|^2\right]+\Expect\left[||\nabla J(\theta_{t})-g_{t}||\Delta_{t}+\frac{1}{2}\|(\nabla^2 J(\theta_{t})-H_{t})\|\Delta_t^2+\frac{M}{6}||d_{t+1}||^{3}\right]\\
   &\leq\Expect[\lambda_t|\lambda_t>0]\cdot \Prob (\|\lambda_t\|>0)\cdot(-\frac{2\epsilon}{M^2})+ \Expect[\lambda_t|\lambda_t=0]\cdot \Prob (\|\lambda_t\|=0)\cdot(-\frac{1}{2}d_k^2)+\frac{5\epsilon^{3/2}}{3M^2}\\
                    &= 		\Expect[\lambda_t|\lambda_t>0]\cdot \Prob(\|\lambda_t\|>0)\cdot(-\frac{2\epsilon}{M^2})+ \frac{5\epsilon^{3/2}}{3M^2}\\
                    &=\Expect[\lambda_t]\cdot(-\frac{2\epsilon}{M^2})+ \frac{5\epsilon^{3/2}}{3M^2}.
			\end{aligned}
		\end{equation}
  Summing from $t=0 $ to $T-1$, and plugging T into the equation, we have 
  \begin{equation}\label{eq3.4}
          \frac{1}{T}\Expect\left[\sum_{t=1}^{T}(J(\theta_{t})-J(\theta_{t-1}))\right] \geq\frac{1}{T} \cdot (J^*-J(\theta_0))\geq-\frac{\epsilon^{3/2}}{24M^2}.
  \end{equation}
  Combining \eqref{eq3.3} and \eqref{eq3.4}, we have
  \begin{equation}\label{eq3.5}
			\begin{aligned}
			\Expect\left[\lambda_{\bar{t}}\right]\leq \frac{41}{48} \sqrt{\epsilon},
			\end{aligned}
		\end{equation}
  where $\bar{t}$ is uniformly sampled from $0, \dots, T-1$. By Assumption \ref{assm:lip-hessian} and Lemma \ref{lm var for hessian}, we have
   $$
\begin{aligned}
-\frac{41}{48}\sqrt{\epsilon} I \preceq-\Expect[\lambda_{\Bar{t}}] I \preceq \Expect[\tilde{H}_{\Bar{t}}] & =\Expect[V_{\Bar{t}} V_{\Bar{t}}{^\trans} H_{\Bar{t}+1} V_{\Bar{t}} V_{\Bar{t}}{^\trans}+\tilde{H}_{\Bar{t}}-V_{\Bar{t}} V_{\Bar{t}}{^\trans} H_{\Bar{t}+1} V_{\Bar{t}} V_{\Bar{t}}{^\trans}] \\
& =\Expect[V_{\Bar{t}} V_{\Bar{t}}{^\trans} H_{\Bar{t}+1} V_{\Bar{t}} V_{\Bar{t}}{^\trans}+V_{\Bar{t}} V_{\Bar{t}}{^\trans}\left(H_{\Bar{t}}-H_{\Bar{t}+1}\right) V_{\Bar{t}} V_{\Bar{t}}{^\trans}] \\
& \preceq \Expect[V_{\Bar{t}} V_{\Bar{t}}{^\trans} H_{\Bar{t}+1} V_{\Bar{t}} V_{\Bar{t}}{^\trans}+\left\|V_{\Bar{t}} V_{\Bar{t}}{^\trans}\left(H_{\Bar{t}}-H_{\Bar{t}+1}\right) V_{\Bar{t}} V_{\Bar{t}}{^\trans}\right\| I] \\
& \preceq \Expect[V_{\Bar{t}} V_{\Bar{t}}{^\trans} H_{\Bar{t}+1} V_{\Bar{t}} V_{\Bar{t}}{^\trans}]+\Expect[\left\|V_{\Bar{t}} V_{\Bar{t}}{^\trans}\right\|\left\|H_{\Bar{t}+1}-H_{\Bar{t}}\right\|\left\|V_{\Bar{t}} V_{\Bar{t}}{^\trans}\right\| I] \\
& =\Expect[\tilde{H}_{\Bar{t}+1}]+\Expect[\left\|H_{\Bar{t}+1}-H_{\Bar{t}}\right\| I] \\
& \preceq \Expect[\tilde{H}_{\Bar{t}+1}]+\Expect[(\|H_{\Bar{t}+1}-\nabla^2 J(\theta_{\Bar{t}+1})\|+M\left\|d_{\Bar{t}+1}\right\|+\|\nabla^2 J(\theta_{\Bar{t}})-H_{\Bar{t}}\|) I] \\
& \preceq \Expect[\tilde{H}_{\Bar{t}+1}]+\frac{25}{12} \sqrt{\epsilon} I,
\end{aligned}
$$
which indicates that $\Expect[\lambda_{\Bar{t}+1}]\leq\frac{141}{48}$. By Lemma \ref{lm var for hessian}, we have
  \begin{equation*}
  \begin{aligned}
        & \Expect[\|\Tilde{\nabla}^2 J(\theta_{\bar{t}+1})-\Tilde{H}_{\bar{t}+1}\|] \leq \Expect[\|V_{\bar{t}+1}V_{\bar{t}+1}^{\trans}\|\cdot\|\nabla^2 J(\theta_{\bar{t}+1})-H_{\bar{t}+1}\|\cdot\|V_{\bar{t}}V_{\bar{t}+1}^{\trans}\|] = \Expect[\|\nabla^2 J(\theta_{\bar{t}+1})-H_{\bar{t}+1}\|] \leq \frac{1}{24} \sqrt{\epsilon}\\
    \iff &\Expect[\Tilde{\nabla}^2 J(\theta_{\bar{t}+1})]\succeq \Expect[\Tilde{H}_{\Bar{t}}]-\frac{1}{24}\sqrt{\epsilon} I\\
    \iff &\Expect[\lambda_{\min}(\Tilde{\nabla}^2 J(\theta_{\bar{t}+1}))] \geq \Expect[\lambda_{\min}(\Tilde{H}_{\Bar{t}+1})]-\frac{1}{24}\sqrt{\epsilon}= -\Expect[\lambda_{\Bar{t}+1}]-\frac{1}{24}\sqrt{\epsilon} \geq -\frac{143}{48}\sqrt{\epsilon}\geq -3\sqrt{\epsilon}. \\
  \end{aligned}
  \end{equation*}

  \end{proof}

\subsection{Proof of Lemma \ref{var-bound for g_DVR-SOPO}}
\begin{proof}
     For ease of presentation, we consider $t<q$. The general case is a straightforward extension.  Recall $\theta(a):=a  \theta_t+(1-a) \theta_{t-1}$. %
\begin{equation}\label{eq:havr-01}
\begin{aligned}
\Expect_t\left\|g_t-\nabla J(\theta_{t})\right\|^2 & = \Expect_t\left\|g_{t-1} + \xi_t-\nabla J(\theta_{t})\right\|^2 \\
& = \Expect_t\left\|\xi_t-(\nabla J(\theta_{t})-\nabla J(\theta_{t-1})) + g_{t-1} -\nabla J(\theta_{t-1}) \right\|^2 \\
& \stackrel{(\clubsuit)}{=}\Expect_t\left\| \xi_t-(\nabla J(\theta_{t})-\nabla J(\theta_{t-1}))  \right\|^2+\Expect_t\|g_{t-1}-\nabla J(\theta_{t-1})\|^2
\end{aligned}
\end{equation}
where $(\clubsuit)$ is due to the fact that $\xi_t$ is conditionally independent of $g_{t-1},\theta_{t-1}$.

By our construction~\eqref{eq:xi_t}, $\xi_t$ is an unbiased estimator of $\nabla J(\theta_{t})-\nabla J(\theta_{t-1})$.
Next we bound the variance of estimator $\xi_t$. 
Let us denote $m=|\widehat{\Mcal}_g|$ and consider samples $a_1,\tau_1, \ldots, a_m,\tau_m$. 
Then we have
\[
\begin{aligned}
& \Expect_t\left\| \xi_t-(\nabla J(\theta_{t})-\nabla J(\theta_{t-1}))  \right\|^2 \\
& = \Expect\left[\left\|\left(\frac{1}{m} \sum_{i=1}^m H(\theta(a_i),\tau(a_i)) - \int_0^1 \nabla^2 J(\theta(a)) \del a \right) \cdot v\right\|^2\right] \\
& = \frac{1}{m^2} \sum_{i=1}^m \Expect\left[\left\|\left( H(\theta(a_i),\tau(a_i)) - \int_0^1 \nabla^2 J(\theta(a)) \del a \right) \cdot v\right\|^2\right] \\
& = \frac{1}{m^2} \sum_{i=1}^m \Expect\left[\left\|\left( H(\theta(a_i),\tau(a_i)) - \nabla^2 J(\theta(a_i)) +  \int_0^1 \left(\nabla^2J(\theta(a_i))- \nabla^2 J(\theta(a))\right) \del a \right) \cdot v\right\|^2\right] \\
& = \frac{1}{m^2} \sum_{i=1}^m \Expect\left[\left\|\left( H(\theta(a_i),\tau(a_i)) - \nabla^2 J(\theta(a_i))  \right) \cdot v\right\|^2\right] +\frac{1}{m^2} \sum_{i=1}^m \Expect \left[\left\|\int_0^1 \left(\nabla^2J(\theta(a_i))- \nabla^2 J(\theta(a))\right) \del a \cdot v\right\|^2\right]\\
& = \frac{1}{m^2} \sum_{i=1}^m \Expect\left[\left\|H(\theta(a_i),\tau(a_i)) - \nabla^2 J(\theta(a_i)) \right\|^2\right] \|v\|^2 +\frac{1}{m^2} \sum_{i=1}^m \Expect \left[\int_0^1 \left\|\nabla^2J(\theta(a_i))- \nabla^2 J(\theta(a))\right\|^2\del a \right] \|v\|^2\\
& \le \frac{G_H^2}{m} \|v\|^2 + \frac{M^2}{m}\|v\|^4,
\end{aligned}
\]
where the inequality uses the Hessian variance bound in Lemma~\ref{lm LG-Variance-bound} and the Hessian Lipschitzness (Assumption~\ref{assm:lip-hessian}) to bound the two sums, respectively.
Due to the trust region choice, we have $\|v\|=\left\|\theta_{t+1}-\theta_{t}\right\|\leq \frac{2\sqrt{\epsilon}}{M}$. It follows that
\[
\Expect_t\left\| \xi_t-(\nabla J(\theta_{t})-\nabla J(\theta_{t-1}))  \right\|^2 \le \frac{4G_H^2\epsilon}{m M^2} + \frac{16\epsilon^2}{m M^2}.
\]

Since $\epsilon\leq \frac{G_H^2}{4}$, by telescoping \eqref{eq:havr-01} and taking the expectation over all the randomness, we obtain
$$
\begin{aligned}
\Expect\left\|g_t-\nabla J\left(\theta_{t}\right)\right\|^2 & \leq \frac{8t\cdot G_H^2\epsilon}{\left|\widehat{\Mcal}_g\right|M^2}+\Expect \left\|g_0-\nabla J\left(\theta_0\right)\right\|^2  \leq \frac{8q\cdot G_H^2\epsilon}{\left|\widehat{\Mcal}_g\right|M^2}+\frac{G_g^2}{\left|\mathcal{M}_0\right|}.
\end{aligned}
$$
By setting $q=\frac{1}{8 \epsilon^{1/2}},\left|\widehat{\Mcal}_g\right|=\frac{288G_H^2}{M^2 \epsilon^{3/2}}$, and $\left|\mathcal{M}_0\right|=\frac{288G_g^2}{ \epsilon^2}$, the result of the lemma follows.
\end{proof}
\subsection{Proof of Theorem \ref{thm conv-DVR}}
The proof of Theorem~\ref{thm conv-DVR} is nearly identical to that of Theorem~\ref{thm sgd}. The only difference is that the  bound for the variance of gradient estimator is based on Lemma~\ref{var-bound for g_DVR-SOPO}.
\newpage

\section{Full dimension trust region method for policy optimization}\label{sec:FDTR}
Our proposed dimension-reduced method aims to alleviate the high computation cost associated with solving the trust region subproblem in the whole space. However, it should be noted that by invoking the full dimension trust region subproblem, we can achieve a second-order stationary point (SOSP) in the whole space with a sample complexity of $\mathcal{O}(\epsilon^{-3})$, although this method requires more Hessian-vector products in the subsolver.

Specifically, we need to solve the full-dimension trust region (FDTR) problem:
\begin{equation}\label{eq:FRTR}
\begin{aligned}
&\min_{\alpha \in \mathbb{R}^n} && m_t(d):=J\left(\theta_t\right)+g_t^\trans d+\frac{1}{2} d^\trans H_t d \\
&\   \textup{ s.t.} && \|d\| \leq \Delta,
\end{aligned}
\end{equation}
and the vector $d_t$ is the global solution to FDTR problem~\eqref{eq:FRTR} if it is feasible and there exists a Lagrange multiplier $x_t \geq 0$ such that $\left(d_t, x_t\right)$ is the solution to the following equations:
\begin{equation}\label{optimal condi for full original}
    \left(H_t+\lambda \right) d+g_t=0, Q_t+xI \succeq 0,x\left(\Delta-\|d\|\right)=0 .
\end{equation}
To solve the full-dimension trust-region subproblem, the Steihaug-CG method~\citep{steihaug1983conjugate} and dogleg method can be used.
If we can solve the subproblem efficiently, by substituting $\lambda_t$ of  in the proof of Theorem \ref{thm sgd} with $x_t$ and following a similar procedure, we can prove the following theorem:
\begin{theorem}[Convergence rate of full-dimension SOPO]\label{thm full sgd}
     Suppose Assumptions~\ref{assm:bdd-reward}-\ref{assm:lip-hessian} hold.
     Let $\Delta_t=\Delta=\frac{2\sqrt{\epsilon}}{M}$ and $\Delta_J$ be a constant number that s.t. $\Delta_J\geq J(\theta_0)-J^*$,
     by setting $M_g=\frac{144G_g^2}{\epsilon^2}$,  $\left|\mathcal{M}_H\right|=\frac{22\times24^2G_H^2\log(d)}{\epsilon}$, $T=\frac{24M^2\Delta_J}{ \epsilon^{\frac{3}{2}}}$ in Algorithm~\ref{alg:drsopo}, we have
$$
\Expect[\left\|\nabla J\left(\theta_{\bar{t}}\right)\right\|] \leq \frac{3}{M}\epsilon ,\\
\Expect[\lambda_{\min}(\nabla^2J(\theta_{\bar{t}}))]\geq-2\sqrt{\epsilon},
$$
where $\bar{t}$ is sampled from $\{1, \ldots, T\}$ uniformly  at random.
Moreover, with probability at least $\frac{7}{8}$,  we have $$
\left\|\nabla J\left(\theta_{\bar{t}}\right)\right\| \leq \frac{12}{M}\epsilon.
$$
\end{theorem}

\begin{theorem}[Convergence rate of full-dimension VRSOPO]\label{thm full conv-DVR}
     Suppose Assumptions \ref{assm:bdd-reward}-\ref{assm:lip-hessian} hold. Let $\Delta_t=\Delta=\frac{2\sqrt{\epsilon}}{M}$ and $\Delta_J$ be a constant number that s.t. $\Delta_J\geq J(\theta_0)-J^*$. If we set $\epsilon\leq \frac{G_H^2}{4}$, $q=\frac{1}{8 \epsilon^{1/2}},|\widehat{\Mcal}_g|=\frac{288G_H^2}{M^2 \epsilon^{3/2}}$, $\left|\mathcal{M}_0\right|=\frac{288G_g^2}{\epsilon^2}$, $\left|\mathcal{M}_H\right|=\frac{22\times24^2G_H^2\log(d)}{\epsilon}$, and $T=\frac{24M^2\Delta_J}{ \epsilon^{\frac{3}{2}}}$ in Algorithm~\ref{alg:dvrsopo}, then we have
\[
\Expect[\left\|\nabla J\left(\theta_{\bar{t}}\right)\right\|] \leq \frac{3}{M}\epsilon ,\\
\Expect[\lambda_{\min}(\tilde{\nabla}^2J(\theta_{\bar{t}}))]\geq-2\sqrt{\epsilon},
\]
where $\bar{t}$ is uniformly sampled from $\{1, \ldots, T\}$.
Moreover, with probability at least $\frac{7}{8}$,  we have $$
\left\|\nabla J\left(\theta_{\bar{t}}\right)\right\| \leq \frac{12}{M}\epsilon.
$$
\end{theorem}
\newpage
\section{Infinite setting}\label{sec:infinite}
In the infinite setting, where the objective function $J(\theta)$ is actually a truncated objective, and $J_{\infty}(\theta)$ denotes the infinite setting objective. We claim that if we set $H=\mathcal{O}(\log(\epsilon^{-1}))$, then finding an SOSP for the truncated objective is sufficient for finding an SOSP for the infinite setting objective.
\begin{lemma} \label{infinite setting}
There exists $D$, $D^{'}$ such that for all $\theta \in \reals^n$, we have
    \begin{equation*}       
    \|J_{\infty}(\theta)-J(\theta)\|\leq \frac{R}{1-\gamma}\gamma^H,\quad \|\nabla J_{\infty}(\theta)-\nabla J(\theta)\|\leq D \gamma^H ,\quad\|\nabla^2 J_{\infty}(\theta)-\nabla^2 J(\theta)\| \leq D^{'} \gamma^H.
    \end{equation*}
\end{lemma}
\begin{proof}
The first two inequality can be found in \cite{yuan2022general}, where $D=\frac{GR}{1-\gamma}\sqrt{\frac{1}{1-\gamma}+H}$. For the last inequality, we have
    \begin{equation*}
    \begin{aligned}
     & \|\nabla^2 J_{\infty}(\theta)-\nabla^2  J(\theta)\|^2 \\
     &= \left\|\Expect_\tau \left[\sum_{h=H}^{\infty} \gamma^h r\left(s_h, a_h\right)u_h\cdot u_h^\trans +\sum_{h=H}^{\infty}\gamma^h r\left(s_h, a_h\right)\sum_{t^{\prime}=0}^h\nabla_{{\theta}}^2 \log \pi_{{\theta}}(a_{t^{\prime}} \mid s_{t^{\prime}})\right]\right\|^2 \\  
     &\leq \Expect_{\tau}\left[\left\|\sum_{h=H}^{\infty} \gamma^h r\left(s_h, a_h\right)u_h\cdot u_h^\trans\right\|^2\right]+\Expect_\tau\left[\left\|\sum_{h=H}^{\infty}\gamma^h r\left(s_h, a_h\right)\sum_{t^{\prime}=0}^h\nabla_{{\theta}}^2 \log \pi_{{\theta}}(a_{t^{\prime}} \mid s_{t^{\prime}})\right\|^2\right]\\
     &\leq \Expect_\tau\left[\left(\sum_{h=H}^{\infty} \gamma^h R^2\right)\left(\sum_{h=H}^{\infty} \gamma^h\|u_h\cdot u_h^\trans \|^2\right)\right]+\Expect_\tau\left[\left(\sum_{h=H}^{\infty} \gamma^h R^2\right)\sum_{h=H}^{\infty} \gamma^h\left\|\sum_{t^{\prime}=0}^h\nabla_{{\theta}}^2 \log \pi_{{\theta}}(a_{t^{\prime}} \mid s_{t^{\prime}})\right\|^2\right]\\
     &\leq \frac{\gamma^H R^2}{1-\gamma}\sum_{h=H}^{\infty}\gamma^h(h+1)^3\sum_{t^{\prime}=0}^h\Expect_\tau\bsbra{\|\nabla_{{\theta}} \log \pi_{{\theta}}\left(a_t \mid s_t\right)\|^4}+\frac{\gamma^H R^2}{1-\gamma}\sum_{h=H}^{\infty}\gamma^h(h+1)\sum_{t^{\prime}=0}^h\Expect_\tau\bsbra{\|\nabla_{{\theta}}^2 \log \pi_{{\theta}}(a_{t^{\prime}} \mid s_{t^{\prime}})\|^2}\\
     &=\frac{\gamma^{2H} R^2G^4}{1-\gamma}\sum_{h=0}^{\infty}\gamma^h(h+1+H)^4+ \frac{\gamma^{2H} R^2L^2}{1-\gamma}\sum_{h=0}^{\infty}\gamma^h(h+1+H)^2\\
     &\leq\frac{\gamma^{2H} R^2G^4}{(1-\gamma)^2}\left[\frac{24}{(1-\gamma)^4}+\frac{24H}{(1-\gamma)^3}+\frac{12H^2}{(1-\gamma)^2}+\frac{4H^3}{(1-\gamma)}+H^4\right]+\frac{\gamma^{2H} R^2L^2}{(1-\gamma)^2}\left[\frac{2}{(1-\gamma)^2}+\frac{2H}{(1-\gamma)}+H^2\right],
    \end{aligned}
\end{equation*}
where the first inequality is due to Jensen's inequality, the second inequality is due to Cauchy-Schwarz inequality and the last inequality is due to Inequality of arithmetic and geometric means. It follows that
\[D^{'}=\frac{RG^2}{1-\gamma}\sqrt{\frac{24}{(1-\gamma)^4}+\frac{24H}{(1-\gamma)^3}+\frac{12H^2}{(1-\gamma)^2}+\frac{4H^3}{(1-\gamma)}+H^4}+\frac{RL}{1-\gamma}\sqrt{\frac{2}{(1-\gamma)^2}+\frac{2H}{(1-\gamma)}+H^2}.\]
\end{proof}

\newpage

\section{Practical versions of DR-SOPO and DVR-SOPO}\label{Practical versions of DR-SOPO and DVR-SOPO}

\begin{algorithm}[H]
	\caption{Practical DR-SOPO algorithm}\label{prac-drsopo}
	\begin{algorithmic}[1]
        \STATE Given $T$, $\Delta_1$, $\Delta_{max}$, $\eta$
		\FOR{$t=1, \dots ,T$}
            \STATE Collect sample trajectories $\Mcal_g$ and compute $g_t$
            \STATE Collect sample trajectories $\Mcal_H$ and compute $H_t$
            \STATE Compute stepsize $\alpha = (\alpha_1,\alpha_2)$ by solving the radius-free problem~\eqref{radius-free}
            \IF{$\|\alpha\| > \Delta_{max}$}
                \STATE $\alpha = \alpha / \|\alpha\| * \Delta_{max}$
            \ENDIF
            \STATE Calculate $\rho_t$ by ~\ref{eqDRSOM RATIO}
            \IF{$\rho_k > \eta$}
                \STATE Update: $\theta_{t+1} \gets \theta_{t}-\alpha_1 g_t+\alpha_2 d_t$
            \ELSE \STATE Adjust the Lagrange multiplier $\lambda_t$ 
            \ENDIF
        \ENDFOR
        \STATE return $\theta_{\Bar{t}}$, which is uniformly picked from $\{\theta_t\}_{t=1,\cdots,T}$
	    \end{algorithmic}
\end{algorithm}

\begin{algorithm}[H]
	\caption{Practical DVR-SOPO algorithm} 
	\label{prac-dvrsopo}
	\begin{algorithmic}[1]
       \STATE Given $T$, $\Delta_1$, $\Delta_{max}$, $\eta$, $q$
		\FOR{$t=1, \dots ,T$}
        \IF{$\mod (t,q)=0$}
	    \STATE Collect sample trajectories $\Mcal_0$ and compute $g_t$ 
        \ELSE
        \STATE Collect sample  trajectories $\widehat{\Mcal}_g$ and compute $\xi_t$:
        \[g_t=g_{t-1}+\Delta_t\]
        \ENDIF
        \STATE sample $|\Mcal_H|$ trajectories to construct $H_t$
        \STATE Compute stepsize $\alpha = (\alpha_1,\alpha_2)$ by solving the radius-free problem (\ref{radius-free})
        \IF{$\|\alpha\| > \Delta_{max}$}
            \STATE $\alpha = \alpha / \|\alpha\| * \Delta_{max}$
        \ENDIF
        \STATE Calculate $\rho_t$ by ~\ref{eqDRSOM RATIO}
        \IF{$\rho_k > \eta$}
            \STATE Update: $\theta_{t+1} \gets \theta_{t}-\alpha_1 g_t+\alpha_2 d_t$
        \ELSE \STATE Adjust the lagrange multiplier $\lambda_t$
        \ENDIF
        \ENDFOR
        \STATE return $\theta_{\Bar{t}}$, which is uniformly picked from $\{\theta_t\}_{t=1,\cdots,T}$
	    \end{algorithmic}
\end{algorithm}

\newpage
\section{Hyper-parameter Settings}\label{Parameter Settings}
\begin{table}[ht]
\begin{minipage}{\textwidth}
\small
    \caption{Hyper-parameter Settings}\label{tab:Hyper-parameter}
    \centering
    \begin{tabular}{|c|c|c|c|c|}
        \hline
        Environment & Swimmer & Walker2d & HalfCheetah & Ant \\
        \hline
        Horizon & 500 & 500 & 500 & 500 \\
        \hline
        Baseline & Linear & Linear & Linear & Linear \\
        \hline
        Number of timesteps & $10^7$ & $10^7$ & $10^7$ & $10^7$ \\
        \hline
        NN size & 64 $\times$ 64 & 64 $\times$ 64 & 64 $\times$ 64 & 64 $\times$ 64 \\
        \hline
        NN activation function & Tanh & Tanh & Tanh & Tanh \\
        \hline
        HAPG learning rate & 0.01 & 0.01 & 0.01 & 0.01 \\
        \hline
        REINFORCE learning rate & 0.01 & 0.01 & 0.01 & 0.01 \\
        \hline
        REINFORCE $\Mcal_g$ & 50 & 50 & 50 & 50 \\
        \hline 
        HAPG $\Mcal_0$ & 50 & 50 & 50 & 50 \\
        \hline
        HAPG $\widehat{\Mcal}_g$ & 10 & 10 & 10 & 10 \\
        \hline
        DR-SOPO $\Mcal_g$ & 50 & 50 & 50 & 50 \\
        \hline
        DR-SOPO $\Mcal_H$ & 10 & 10 & 10 & 10 \\
        \hline
        DVR-SOPO $\Mcal_0$ & 50 & 50 & 50 & 50 \\
        \hline
        DVR-SOPO $\widehat{\Mcal}_g$ & 10 & 10 & 10 & 10 \\
        \hline
        DVR-SOPO $\Mcal_H$ & 10 & 10 & 10 & 10 \\
        \hline
        HAPG $q$ & 10 & 5 & 5 & 5 \\
        \hline
        DVR-SOPO $q$ & 10 & 5 & 5 & 5 \\
        \hline
        DR-SOPO $\Delta_{max}$ & 2 & 0.2 & 0.02 & 0.05 \\
        \hline
        DVR-SOPO $\Delta_{max}$ & 2 & 0.2 & 0.02 & 0.05 \\
        \hline
        DR-SOPO $\eta$ & 0.001 & 0.001 & 0.001 & 0.001 \\
        \hline
        DVR-SOPO $\eta$ & 0.001 & 0.001 & 0.001 & 0.001 \\
        \hline
        $\mu$ \footnote{Parameter $\mu$ is discussed in~\eqref{Biased Hessian Estimator}, and set as the same for all the four algorithms.} & $\frac{1}{500}$ & $\frac{1}{500}$ & $\frac{1}{500}$ & $\frac{1}{500}$ \\
        \hline
    \end{tabular}
    \end{minipage}
\end{table}

\end{document}